\documentclass[11pt]{amsart}
\usepackage{amsmath,amssymb,amsthm,mathtools}
\usepackage[T1]{fontenc}
\usepackage[utf8]{inputenc}
\usepackage[margin=1in]{geometry}
\usepackage{xcolor}
\usepackage{hyperref}
\usepackage[nameinlink,noabbrev]{cleveref}
\usepackage{enumitem}

\hypersetup{
  colorlinks=true,
  linkcolor=blue!50!black,
  citecolor=blue!50!black,
  urlcolor=blue!50!black,
  pdftitle={Split-prime supercongruence at the mixed CM point},
  pdfauthor={Alex Shvets}
}

\newtheorem{theorem}{Theorem}[section]
\newtheorem{proposition}[theorem]{Proposition}
\newtheorem{lemma}[theorem]{Lemma}
\newtheorem{corollary}[theorem]{Corollary}

\theoremstyle{definition}
\newtheorem{definition}[theorem]{Definition}
\theoremstyle{remark}
\newtheorem{remark}[theorem]{Remark}
\newtheorem*{theorem*}{Theorem}

\crefname{theorem}{theorem}{theorems}
\Crefname{theorem}{Theorem}{Theorems}
\crefname{proposition}{proposition}{propositions}
\Crefname{proposition}{Proposition}{Propositions}
\crefname{lemma}{lemma}{lemmas}
\Crefname{lemma}{Lemma}{Lemmas}
\crefname{corollary}{corollary}{corollaries}
\Crefname{corollary}{Corollary}{Corollaries}
\crefname{remark}{remark}{remarks}
\Crefname{remark}{Remark}{Remarks}
\crefname{section}{section}{sections}
\Crefname{section}{Section}{Sections}

\newcommand{\Q}{\mathbf{Q}}
\newcommand{\Z}{\mathbf{Z}}

\newcommand{\dd}{\mathrm{d}}
\newcommand{\ord}{\operatorname{ord}}
\newcommand{\Res}{\operatorname{Res}}
\newcommand{\Tr}{\operatorname{Tr}}
\newcommand{\Span}{\operatorname{Span}}
\newcommand{\vp}{v_p}
\newcommand{\Fp}{\mathbf{F}_p}

\title[Split-prime supercongruence at the mixed CM point]{Split-prime supercongruence at the mixed CM point $(\tfrac16,\tfrac13;1)$}
\author{Alex Shvets}
\thanks{ORCID: 0009-0005-9802-379X}
\address{Haifa, Israel}
\email{alex@shvets.io}
\urladdr{https://shvets.io}
\date{}

\begin{document}
\begin{abstract}
For the mixed CM point $(a,b,c)=(\tfrac16,\tfrac13,1)$, set
$A_n^{\mathrm{mix}}:=108^n[z^n]{}_2F_1(\tfrac16,\tfrac13;1;z)^3$. We prove unconditionally that for every split prime $p\ge 7$, $p\equiv 1\pmod 3$, and every $m\ge 1$,
\[
A_{mp}^{\mathrm{mix}}\equiv A_m^{\mathrm{mix}}\pmod{p^4}.
\]
The exponent $4$ exceeds the generic weight-$3$ Hodge-gap prediction of $3$ from the Roberts--Rodriguez Villegas framework~\cite{RRV}; the additional factor of $p$ is a CM enhancement traced to the $j=0$ elliptic curve. We also prove the matching unconditional inert-prime obstruction. For inert primes $p\ge 5$, equivalently $p\equiv 2\pmod 3$, the split-style formal-parameter congruence cannot hold on the $q$-side, and the coefficient sequence itself satisfies an explicit Cartier parity law modulo $p$: one Frobenius step exchanges the two Eisenstein branches and two steps return to the original branch.

The proof proceeds via the modular realization on $\Gamma_0(3)$ with $C_{\mathrm{mix}}=3E_{5,\chi_0,\chi_3}-27E_{5,\chi_3,\chi_0}$ and parameter $t=u/(1+27u)^2$, $u=\eta(3\tau)^{12}/\eta(\tau)^{12}$. Lagrange--B\"urmann reduces the supercongruence to three Cartier identities $\Lambda_p(C_{\mathrm{mix}}U_p^\ell)\equiv 0\pmod{p^4}$, $\ell=1,2,3$. The key technical input is a length-three Witt--Cartier pole estimate at the tame elliptic stack point $P_-:u=-1/27$, $j=0$, where $\mu_3$-equivariance of the canonical Frobenius lift, combined with $p\equiv 1\pmod 3$, forces $K\equiv 1\pmod 3$ in the residue calculation and produces the extra digit. A short Atkin--Lehner intertwining transports the resulting classicality from $C_0$ to $uC_0$ and yields the unconditional mixed Cartier cancellation, which is the exact Katz--Dwork input for the main theorem.
\end{abstract}

\maketitle

\section{Introduction}\label{sec:intro}

Let
\[
A_n^{\mathrm{mix}}:=108^n[z^n]{}_2F_1\!\left(\tfrac16,\tfrac13;1;z\right)^3.
\]
The sequence begins
\[
1,\ 18,\ 864,\ 55152,\ 4035906,\ 320012532,\ 26749991016,\ldots .
\]
The main result is the following.

\begin{theorem*}[Main theorem]
For every split prime $p\ge 7$, $p\equiv 1\pmod 3$, and every $m\ge 1$,
\[
A_{mp}^{\mathrm{mix}}\equiv A_m^{\mathrm{mix}}\pmod{p^4}.
\]
The proof establishes the scalar Katz--Dwork congruence
\[
\Lambda_p\!\left(C_{\mathrm{mix}}(q)H_{\mathrm{mix}}(q)^{pm}\right)\equiv C_{\mathrm{mix}}(q)H_{\mathrm{mix}}(q)^m\pmod{p^4}
\]
for every $m\ge 1$, from which the displayed coefficient congruence follows by Lagrange--B\"urmann.
\end{theorem*}

This is proved as \Cref{thm:split-final-closed}. The exponent $4$ is one more than the generic weight-$3$ Hodge-gap exponent predicted by the Roberts--Rodriguez Villegas framework~\cite{RRV}; the extra factor of $p$ is a CM enhancement attached to the $j=0$ split-ordinary elliptic curve. The same level of enhancement appears in earlier work for related but distinct families: Long--Ramakrishna obtain $p^6$ for a Van Hamme-type sum at the $j=0$ point \cite{LR}, and Zudilin records analogous CM-driven gains \cite{Zudilin-HRJ}. For the present sequence in this exact form, the result appears to be new.

The matching negative result is the inert-prime obstruction.

\begin{theorem*}[Inert obstruction]
If $p\ge 5$ and $\chi_3(p)=-1$, the $p^4$ formal-parameter congruence cannot hold. More precisely, for $m=p^am_0$ with $p\nmid m_0$ and $\beta(m_0)\not\equiv 0\pmod p$,
\[
\vp(c_{mp}^{\mathrm{mix}}-c_m^{\mathrm{mix}})=0.
\]
Thus the obstruction occurs already on the $q$-side, in the Eisenstein coefficient tower.
\end{theorem*}

There is also a coefficient-level first-digit theorem. For $\varepsilon\in\{\pm1\}$ define
\[
A_m^{(\varepsilon)}:=[q^m]\bigl(C_0(q)-27\varepsilon\,u(q)C_0(q)\bigr)H_{\mathrm{mix}}(q)^m.
\]
Then for every prime $p\ge5$, every $r\ge0$, and every $m\ge0$,
\[
A_{mp^r}^{\mathrm{mix}}\equiv A_m^{(\chi_3(p)^r)}\pmod p.
\]
In particular, if $p\equiv2\pmod3$, then
\[
A_{mp^{2r}}^{\mathrm{mix}}\equiv A_m^{\mathrm{mix}}\pmod p,
\qquad
A_{mp^{2r+1}}^{\mathrm{mix}}\equiv A_m^{(-1)}\pmod p,
\]
and $A_p^{\mathrm{mix}}\equiv72\pmod p$.

These are proved as \Cref{thm:inert-obstruction,prop:no-inert-dwork,thm:inert-cartier-parity,cor:inert-Ap}. The dichotomy comes from the second Eisenstein piece $E_{5,\chi_3,\chi_0}$: its Euler factor is compatible with the $p^4$ tower at split primes and acquires a unit-root factor at inert primes; on the coefficient sequence this is seen as branch exchange modulo $p$.

\subsection*{Method}

The modular foundation is the order drop $L^{(3)}=(S-108)L^{(2)}$ for the symmetric-cube recurrence specialized at $(1/6,1/3;1)$ from the Mao--Tian framework \cite{MT}, together with the level-$3$ dictionary
\[
u(\tau)=\frac{\eta(3\tau)^{12}}{\eta(\tau)^{12}},\qquad
t=\frac{u}{(1+27u)^2},\qquad
C_{\mathrm{mix}}=(1-27u)C_0,
\]
where $C_0=3E_{5,\chi_0,\chi_3}$ and $uC_0=E_{5,\chi_3,\chi_0}$. Lagrange--B\"urmann gives $A_m^{\mathrm{mix}}=[q^m]C_{\mathrm{mix}}H_{\mathrm{mix}}^m$ with $H_{\mathrm{mix}}=q/t$, and a Katz--Dwork reduction converts the supercongruence to three Cartier identities
\[
\Lambda_p(C_{\mathrm{mix}}U_p^\ell)\equiv 0\pmod{p^4},\qquad \ell=1,2,3,
\]
where $U_p=\log(t(q)^p/t(q^p))\in pq\Z_{(p)}[[q]]$.

The main estimate is a length-three Witt--Cartier pole estimate at the tame elliptic stack point $P_-:u=-1/27$, $j=0$, with effective $\mu_3$-stabilizer (\Cref{prop:pole-estimate-star}). Working in the formal completion $[\operatorname{Spf}\Z_p[\zeta_3][[r]]/\mu_3]$ and using the descended canonical Frobenius lift constructed via a fine prime-to-$3$ atlas, $\mu_3$-equivariance forces the residue index $K$ to satisfy $K\equiv 1\pmod 3$. For $K\ge 2$ this gives $K\ge 4$, which after the valuation count carried out in \Cref{prop:pole-estimate-star} yields the desired $p^{4-\ell}$ bound. The same estimate, applied to both $C_0$ and $uC_0$, gives classicality of $g\Lambda_p(C_0\phi_p^\ell/p^\ell)$ and $g\Lambda_p(uC_0\phi_p^\ell/p^\ell)$ as weight-$5$, $\chi_3$ forms modulo $p^{4-\ell}$.

A short Atkin--Lehner intertwining $T_p w_3^*=\chi_3(p)w_3^*T_p$ then converts the joint classicalities into the bridge identity $\beta_\ell=\gamma_\ell/27$ between the scalars on the two sides; at split primes $\chi_3(p)=1$ the intertwining is untwisted. This cancels the two Eisenstein components in $C_{\mathrm{mix}}=C_0-27uC_0$ and gives the unconditional mixed Cartier identities $\Lambda_p(C_{\mathrm{mix}}U_p^\ell)\equiv 0\pmod{p^4}$ for $\ell=1,2,3$, which is the exact Katz--Dwork input for the main theorem.

\subsection*{Notes on the literature}

The first-digit congruence
$A_p^{\mathrm{mix}}\equiv 18\pmod{p^3}$
that follows from our main theorem is, via Clausen's identity and the substitution $z\mapsto 108t$, equivalent to a special case of Beukers \cite[Ex.~5.2]{Beukers-mod} on the Clausen sum $\sum_{n=0}^{p-1}(3n)!(2n)!/((n!)^5\,108^n)\bmod p^3$; analogous Beukers-method results for binomial sums at other parameter specializations appear in Z.-H. Sun and D. Ye \cite{SunYe}. Our argument differs in route, not in conclusion at this depth. The novelty is the $p^4$-strengthening at all $m\ge 1$, the inert-prime obstruction, and the local stack-level mechanism producing the extra digit.

\subsection*{Outline}

\Cref{sec:modular} records the modular setup and the order drop. \Cref{sec:KD} performs the Katz--Dwork reduction to three Cartier identities. \Cref{sec:WC} proves the Witt--Cartier estimate and the Atkin--Lehner bridge leading to mixed Cartier cancellation. The main theorem (\Cref{thm:split-final-closed}) follows directly from this cancellation via Katz--Dwork reduction. \Cref{sec:inert} proves both the unconditional inert $q$-side obstruction and the coefficient-level inert Cartier parity law modulo $p$. \Cref{sec:comp} summarizes the computational verification.

\subsection*{Notation and central objects}\label{subsec:notation}

We fix the notation used below.

\emph{Sequence and hypergeometric function.} Set
\[
A_n^{\mathrm{mix}}:=108^n[z^n]\,{}_2F_1\!\left(\tfrac16,\tfrac13;1;z\right)^3,\qquad
F_{\mathrm{mix}}(t):=\sum_{n\ge 0}A_n^{\mathrm{mix}}t^n={}_2F_1\!\left(\tfrac16,\tfrac13;1;108t\right)^3.
\]
The factor $108$ is part of the definition of $F_{\mathrm{mix}}$. Apart from the defining coefficient extraction for $A_n^{\mathrm{mix}}$, the symbol $F_{\mathrm{mix}}$ always denotes the rescaled generating function in the variable $t$; the unrescaled cube ${}_2F_1(\tfrac16,\tfrac13;1;z)^3$ appears explicitly only inside the recurrence proof of \Cref{thm:orderdrop}, where it is written out as a function of $z$ without invoking the symbol $F_{\mathrm{mix}}$.
 
\emph{Modular parameters on $\Gamma_0(3)$.} Let $\eta(\tau)=q^{1/24}\prod_{n\ge 1}(1-q^n)$ with $q=e^{2\pi i\tau}$. Define
\[
u(\tau):=\frac{\eta(3\tau)^{12}}{\eta(\tau)^{12}},\quad
g(\tau):=1+27u(\tau),\quad
\alpha(\tau):=\frac{27u}{1+27u},\quad
t(\tau):=\frac{u}{(1+27u)^2}.
\]
The Fricke involution on $X_0(3)$ is $w_3:\tau\mapsto -1/(3\tau)$, acting on $u$ by $w_3:u\mapsto 1/(729u)$, and on $q$-series by pullback $w_3^*$.

\emph{Theta data.} Set $\mathcal A(q):=\sum_{m,n\in\Z}q^{m^2+mn+n^2}$, $\mathcal B(q):=\eta(\tau)^3/\eta(3\tau)$, $\mathcal D(q):=3\eta(3\tau)^3/\eta(\tau)$. These satisfy $\mathcal A^3=\mathcal B^3+\mathcal D^3$ and $\mathcal D^3/\mathcal B^3=27u$ (Borwein \cite{BBG}).

\emph{Characters and Eisenstein series.} Let $\chi_0$ denote the trivial character mod $3$, and $\chi_3$ the unique nontrivial Dirichlet character mod $3$, so $\chi_3(n)=\bigl(\tfrac{n}{3}\bigr)$ for $\gcd(n,3)=1$, $\chi_3(n)=0$ otherwise. For an ordered pair of characters $(\chi,\psi)$ of conductors dividing $3$, with $\chi\psi=\chi_3$ and primitive at the indicated component, the weight-$5$ Eisenstein series is
\[
E_{5,\chi,\psi}(q):=c_0+\sum_{n\ge 1}\Bigl(\sum_{d\mid n}\chi(n/d)\psi(d)d^4\Bigr)q^n,
\]
where the constant term $c_0$ is the standard one fixed by the functional equation; concretely $E_{5,\chi_0,\chi_3}=\tfrac{1}{3}+\sum_{n\ge 1}\bigl(\sum_{d\mid n}\chi_3(d)d^4\bigr)q^n$ and $E_{5,\chi_3,\chi_0}=\sum_{n\ge 1}\bigl(\sum_{d\mid n}\chi_3(n/d)d^4\bigr)q^n=q+15q^2+81q^3+\cdots$. The ordering of subscripts in $E_{5,\chi,\psi}$ follows the convention $\chi$ acts on $n/d$, $\psi$ on $d$; this is the opposite of the PARI/GP \verb|mfeisenstein| argument order, and we refer back to this paragraph whenever a numerical cross-check is performed.
 
\emph{The two Eisenstein lines and the mixed direction.} Set
\[
C_0(q):=3E_{5,\chi_0,\chi_3}(q),\qquad uC_0(q):=E_{5,\chi_3,\chi_0}(q),
\]
where the symbol $uC_0$ is a single name for the modular form $E_{5,\chi_3,\chi_0}$; the equality $u(q)\cdot C_0(q)=E_{5,\chi_3,\chi_0}(q)$ as $q$-series is part of \Cref{thm:modular}(iii). The mixed Eisenstein direction is
\[
C_{\mathrm{mix}}(q):=C_0(q)-27\cdot uC_0(q)=3E_{5,\chi_0,\chi_3}-27E_{5,\chi_3,\chi_0}=(1-27u)C_0,
\]
and the Lagrange--B\"urmann kernel is
\[
H_{\mathrm{mix}}(q):=q/t(q).
\]
Both $C_{\mathrm{mix}}$ and $H_{\mathrm{mix}}$ lie in $\Z[[q]]$.

\emph{Frobenius and Cartier operators on $q$-series.} For a Laurent $q$-series $f(q)=\sum_{n\gg-\infty}a_nq^n$ over $\Z_{(p)}$ with finite principal part, set
\[
\Lambda_p f\,(q):=\sum_{n\gg-\infty}a_{pn}q^n,\qquad
\sigma_p f\,(q):=f(q^p)=\sum_{n\gg-\infty}a_nq^{pn}.
\]
On ordinary power series these operators preserve $\Z_{(p)}[[q]]$, and $\Lambda_p\circ\sigma_p=\mathrm{id}$; on the Laurent extension they remain $\Z_{(p)}$-linear, which is the form required for inputs such as $\Lambda_p(C_{\mathrm{mix}}/t^{rp})$ in \Cref{lem:layer-defect}. The Frobenius defects of \Cref{def:Up,def:nu-phi} are then $U_p=\log(t^p/\sigma_p t)$, $V_p=\log(u^p/\sigma_p u)$, $W_p=\log(g^p/\sigma_p g)$, $\nu_p=\sigma_p t/t^p=e^{-U_p}$, $\phi_p=\nu_p-1$. Here $U_p,V_p,W_p,\phi_p\in pq\Z_{(p)}[[q]]$, while $\nu_p\in 1+pq\Z_{(p)}[[q]]$.

\emph{Stack-local data at $P_-$.} The point $P_-\in X_0(3)$ is $u=-1/27$, equivalently $j=0$; its image in the coarse curve $\mathbb P^1_u$ is regular, but the underlying elliptic curve has $j=0$ with $\mu_6$-automorphism group, of which the $\mu_3$-subgroup preserves the canonical subgroup (\Cref{lem:mu3-preserves-canonical-subgroup}). The local parameter on the corresponding tame stack
\[
[\operatorname{Spf}\Z_p[\zeta_3][[r]]/\mu_3]
\]
is denoted $r$, normalized so that $r^3=u+1/27$ on the cover. The descended canonical Frobenius lift is denoted $\Phi$, and the associated normalized weight-$5$ Katz trace operator is denoted $\mathcal C_{\Phi,5}$ (\Cref{lem:trace-operator-w5}). In the local formula one uses $\mu_3$-equivariant generators $e,e'$ of the weight-$5$, $\chi_3$ line bundle on the source and target. The symbol
\[
\psi:=\frac{\Phi^*t/t^p-1}{p}
\]
is the normalized Frobenius defect; on the Tate chart at $\infty$ it specializes to $\phi_p/p$.

\emph{Cartier--Hecke finite differences.} For $f\in\{C_0,uC_0\}$ and $\ell=1,2,3$, the symbol $\mathfrak E_\ell(f)$ denotes the Hecke finite difference defined in \Cref{lem:cartier-hecke-local}; for $f=C_0$ we also write $\mathfrak E_\ell$, and for $f=uC_0$ we write $\mathfrak E'_\ell$. Two distinct comparisons attached to $\mathfrak E_\ell$ are used:
\begin{itemize}[leftmargin=2em]
\item a \emph{local} identity on the ordinary locus, equating $\mathfrak E_\ell(f)$ to the normalized Katz trace $\mathcal C_{\Phi,5}(f\psi^\ell)$ up to an explicit Verschiebung error of valuation $\ge p^{4-\ell}$ (\Cref{lem:cartier-hecke-local});
\item a \emph{global} holomorphy transfer for $g\cdot\mathfrak E_\ell(f)$ as a weight-$5$, $\chi_3$ form modulo $p^{4-\ell}$, combining the local identity with cusp regularity at $\infty$ and $0$ and the pole estimate at $P_-$ (\Cref{lem:holomorphy-transfer}).
\end{itemize}
On $q$-expansions at the cusp $\infty$ the local identity specializes to
\[
\mathfrak E_\ell(f)\equiv\Lambda_p\!\left(f\,\frac{\phi_p^\ell}{p^\ell}\right)\pmod{p^{4-\ell}}.
\]
The closure argument of \Cref{thm:KB-closed} studies the multiplied form $g\,\mathfrak E_\ell(f)$, where $g=1+27u$; this factor is part of the closure step, not of the definition of $\mathfrak E_\ell$. After the triangular conversion from $\phi_p$-layers to $U_p$-layers (\Cref{lem:phi-U-triangular}), the scalars $\alpha_\ell,\beta_\ell,\gamma_\ell$ realizing classicality of $g\mathfrak E_\ell(C_0)$ and $g\mathfrak E_\ell(uC_0)$ in the basis $\{C_0,uC_0\}$ of $M_5(\Gamma_0(3),\chi_3)$ modulo $p^{4-\ell}$ are defined in \Cref{subsec:bridge}.

\section{Modular setup}\label{sec:modular}

\subsection{Order drop}\label{subsec:orderdrop}

We work in the Ore algebra $\Q[n]\langle S\rangle$ with $Sf(n)=f(n+1)$ and $S\,P(n)=P(n+1)S$.

\begin{theorem}[Order drop]\label{thm:orderdrop}
The sequence $A_n=108^n[z^n]{}_2F_1(1/6,1/3;1;z)^3$ satisfies the order-$2$ recurrence
\[
(n+2)^4 A_{n+2}-6(36n^4+198n^3+424n^2+417n+158)A_{n+1}+324(n+1)(2n+1)(3n+2)(6n+5)A_n=0,
\]
with $A_0=1$, $A_1=18$, $A_2=864$. Equivalently, the third-order Mao--Tian recurrence operator $L^{(3)}$ for the cube of ${}_2F_1$ factors as $L^{(3)}=(S-108)L^{(2)}$ in $\Q[n]\langle S\rangle$, where $L^{(2)}=B_0(n)+B_1(n)S+B_2(n)S^2$ has the displayed coefficients.
\end{theorem}

\begin{proof}
The cube $G(z):={}_2F_1(1/6,1/3;1;z)^3$ satisfies a symmetric-cube ODE of order $3$, giving the Mao--Tian recurrence $L^{(3)}A=0$ with the standard coefficients \cite{MT}; here $A_n=108^n[z^n]G(z)$ and the generating function in the rescaled variable $t$ is $F_{\mathrm{mix}}(t)=G(108t)$, as fixed in the Notation. Setting $B_0,B_1,B_2$ as in the statement, a direct Ore computation gives
\[
(S-108)L^{(2)}=-108B_0+(B_0(n+1)-108B_1)S+(B_1(n+1)-108B_2)S^2+B_2(n+1)S^3,
\]
which equals $L^{(3)}$ term by term. Hence $w(n):=L^{(2)}A$ satisfies $w(n+1)=108w(n)$, and direct computation gives $w(0)=324\cdot 10\cdot 1-6\cdot 158\cdot 18+16\cdot 864=0$. Therefore $w\equiv 0$.
\end{proof}

\subsection{Modular realization}\label{subsec:modular-realization}

Define cubic theta and eta data
\[
\mathcal A(q)=\sum_{m,n\in\Z}q^{m^2+mn+n^2},\quad
\mathcal B(q)=\frac{\eta(\tau)^3}{\eta(3\tau)},\quad
\mathcal D(q)=3\frac{\eta(3\tau)^3}{\eta(\tau)},\quad
u(\tau)=\frac{\eta(3\tau)^{12}}{\eta(\tau)^{12}}.
\]

\begin{lemma}[Preliminary theta identity]\label{lem:prelim-theta}
The Eisenstein series $C_0=3E_{5,\chi_0,\chi_3}$ satisfies
\[
C_0=\mathcal A^2\mathcal B^3,
\]
where $\mathcal A$ and $\mathcal B$ are the cubic theta series above.
\end{lemma}

\begin{proof}
We use only the external level-$3$ theta identities. Moy's differential identity \cite{Moy} gives
\[
C_0=\mathcal B^3\,q\frac{d}{dq}\log u,
\qquad
q\frac{d}{dq}\log u=\mathcal A^2.
\]
Hence $C_0=\mathcal A^2\mathcal B^3$.
\end{proof}

\begin{lemma}[Preliminary divisor and order chart]\label{lem:prelim-divisor-chart}
Let $R=\Z_p$ with $p\ge7$, and let $S$ be either $R$ or $R/p^N R$. Let $P_\infty$ be the cusp with $q=0$, let $P_0$ be the other cusp, and let $P_-:1+27u=0$. On the coarse curve $X_0(3)_S\simeq\mathbb P^1_{u,S}$, the following local order table holds:
\[
\begin{array}{c|ccc}
& P_\infty & P_0 & P_- \\ \hline
u & 1 & -1 & 0\\
C_0 & 0 & 1 & 2\ \text{\rm(stack)}\\
uC_0 & 1 & 0 & 2\ \text{\rm(stack)}.
\end{array}
\]
All leading coefficients in these local descriptions are $S$-units.
\end{lemma}

\begin{proof}
By \Cref{lem:prelim-theta},
\[
C_0=\mathcal A^2\mathcal B^3,
\qquad
uC_0=\frac{\mathcal A^2\mathcal D^3}{27},
\qquad
g:=1+27u=\frac{\mathcal A^3}{\mathcal B^3},
\]
using the Borwein identities $\mathcal A^3=\mathcal B^3+\mathcal D^3$ and $27u=\mathcal D^3/\mathcal B^3$ \cite{BBG}. At $P_\infty$, the eta-products give $\mathcal B=1+O(q)$ and $\mathcal D^3/27=q\varepsilon_\infty(q)$ with $\varepsilon_\infty(0)\in R^\times$; hence $u$ has order $1$, while $\mathcal A$ and $\mathcal B$ are units. Thus $C_0$ is a unit and $uC_0$ has order $1$.

At $P_0$, put $v=1/u$. The Fricke transformation of eta-products exchanges $\mathcal B$ and $\mathcal D$ up to a unit. Hence $\mathcal D$ is a unit, $v=27\mathcal B^3/\mathcal D^3$ is a uniformizer, and $\mathcal A$ is a unit. Therefore $C_0=\mathcal A^2\mathcal B^3$ has order $1$, while $uC_0=\mathcal A^2\mathcal D^3/27$ is a unit.

At $P_-$, the functions $\mathcal B$ and $\mathcal D$ are units and $g=1+27u=\mathcal A^3/\mathcal B^3$. Choose the stack uniformizer $r=\mathcal A/\mathcal B$, so $g=r^3$. Then $C_0=r^2\mathcal B^5$ and $u=-1/27+r^3/27$ is a unit, so both $C_0$ and $uC_0$ have stack order $2$. Since $p\ge7$, the constants $27$ and the displayed unit leading coefficients remain invertible over $S$.
\end{proof}

\begin{theorem}[Modular realization]\label{thm:modular}
With $\alpha=27u/(1+27u)$, $t=u/(1+27u)^2$, the following hold.
\begin{enumerate}[label=\textup{(\roman*)}]
\item $\mathcal A^3=\mathcal B^3+\mathcal D^3$, and $108t=4\alpha(1-\alpha)$.
\item ${}_2F_1(1/6,1/3;1;108t)^3=\mathcal A(q)^3=(1+27u)\eta(\tau)^9/\eta(3\tau)^3$.
\item Setting $C_{\mathrm{mix}}(q):=F_{\mathrm{mix}}(t(q))\cdot\frac{q}{t(q)}\cdot\frac{\dd t}{\dd q}$, one has
\[
C_{\mathrm{mix}}=(1-27u)C_0=3E_{5,\chi_0,\chi_3}-27E_{5,\chi_3,\chi_0},
\]
where $C_0=3E_{5,\chi_0,\chi_3}$ and $uC_0=E_{5,\chi_3,\chi_0}$.
\item $H_{\mathrm{mix}}(q):=q/t(q)=\prod_{3\nmid n}(1-q^n)^{12}\cdot(1+27u(q))^2$.
\end{enumerate}
\end{theorem}

\begin{proof}
The cubic theta identity $\mathcal A^3=\mathcal B^3+\mathcal D^3$ and $\mathcal D^3/\mathcal B^3=27u$ give $\alpha=27u/(1+27u)$ and $108t=4\alpha(1-\alpha)$. The quadratic transformation
${}_2F_1(\tfrac16,\tfrac13;1;4x(1-x))={}_2F_1(\tfrac13,\tfrac23;1;x)$
applied at $x=\alpha$, combined with the Borwein cubic identity ${}_2F_1(\tfrac13,\tfrac23;1;\alpha)=\mathcal A$ \cite{BBG}, gives (ii).

For (iii), set $F_0(\tau):=\eta(\tau)^9/\eta(3\tau)^3$ and $C_0(q):=F_0(\tau)\cdot(q/u)(\dd u/\dd q)$. Moy's level-$3$ identity \cite{Moy} gives $C_0=3E_{5,\chi_0,\chi_3}$. From $F_{\mathrm{mix}}=(1+27u)F_0$ and $t=u(1+27u)^{-2}$ a direct differentiation gives
\[
\frac{q}{t}\frac{\dd t}{\dd q}=\frac{q}{u}\frac{\dd u}{\dd q}\cdot\frac{1-27u}{1+27u},
\]
hence $C_{\mathrm{mix}}=(1-27u)C_0$. The identification $uC_0=E_{5,\chi_3,\chi_0}$ follows from a direct Sturm-bound and Eisenstein-space argument. By \Cref{lem:prelim-divisor-chart}, $uC_0$ is holomorphic at $P_\infty$, $P_0$, and $P_-$, with $a_0=0$ and $a_1=1$ at $P_\infty$; away from the cusps and elliptic point this is clear from the eta/theta expression $uC_0=u\mathcal A^2\mathcal B^3$ of \Cref{lem:prelim-theta}. Hence $uC_0\in M_5(\Gamma_0(3),\chi_3)$. The index $[\operatorname{SL}_2(\Z):\Gamma_0(3)]=4$, so the Sturm bound for weight $5$ on $\Gamma_0(3)$ with any character is $\lfloor 5\cdot 4/12\rfloor=1$. Hence any form in $M_5(\Gamma_0(3),\chi_3)$ is determined by its $q$-coefficients $a_0$ and $a_1$, giving $\dim M_5(\Gamma_0(3),\chi_3)\le 2$. The two Eisenstein series $E_{5,\chi_0,\chi_3}$ and $E_{5,\chi_3,\chi_0}$ are linearly independent: $E_{5,\chi_0,\chi_3}$ has $a_0=1/3$, $a_1=1$, while $E_{5,\chi_3,\chi_0}$ has $a_0=0$, $a_1=1$. Hence $\dim M_5(\Gamma_0(3),\chi_3)=2$, the Eisenstein subspace exhausts $M_5$, and $\dim S_5(\Gamma_0(3),\chi_3)=0$. Since $u(q)C_0(q)$ has $a_0=0$, $a_1=1$, the Sturm bound forces $uC_0=E_{5,\chi_3,\chi_0}$ (matching $q$-coefficients up to the Sturm bound, then equality on all of $M_5$). Both expand as $q+15q^2+81q^3+\cdots$. Part (iv) follows from $q/u=\prod_{3\nmid n}(1-q^n)^{12}$ together with $H_{\mathrm{mix}}=(q/u)(1+27u)^2$.
\end{proof}

\subsection{Lagrange--B\"urmann and the Fricke-trace direction}

\begin{theorem}[Lagrange--B\"urmann]\label{thm:LB}
For every $m\ge 0$, $A_m^{\mathrm{mix}}=[q^m]C_{\mathrm{mix}}(q)H_{\mathrm{mix}}(q)^m$.
\end{theorem}

\begin{proof}
$A_m=[t^m]F_{\mathrm{mix}}(t)=\Res_{t=0}F_{\mathrm{mix}}(t)t^{-m-1}\dd t$. Substituting $t=t(q)$ and using $t(q)=q+O(q^2)$,
$A_m=[q^m]F_{\mathrm{mix}}(t(q))\cdot(q/t(q))\cdot(\dd t/\dd q)\cdot H_{\mathrm{mix}}^m=[q^m]C_{\mathrm{mix}}H_{\mathrm{mix}}^m$.
\end{proof}

\begin{remark}[Why the diagonal $C_0-27uC_0$]\label{rem:fricke-trace-line}
The Fricke involution on $X_0(3)$ is $w_3:u\mapsto 1/(729u)$; the parameter $t$ is $w_3$-invariant. The Eisenstein lines are exchanged: $w_3^*C_0=-27uC_0$ and $w_3^*(uC_0)=-C_0/27$. Hence $\omega_{\mathrm{mix}}=C_{\mathrm{mix}}\,\dd q/q=(1-27u)C_0\,\dd q/q$ is the unique Fricke-trace direction in $\Span\{C_0,uC_0\}$ descending to the quotient $\mathbb P^1_t$. The module-wide $p^4$-congruence on the full span $\Span\{C_0,uC_0\}$ is false: at $p=7$ the coefficient $[q^1](\Lambda_7(uC_0H_0^7)-uC_0H_0)=22478120$ has $7$-adic valuation $1$, not $\ge 4$. The mixed congruence is therefore not a module-wide phenomenon but a feature of the specific Fricke-trace direction.
\end{remark}

\section{Katz--Dwork reduction}\label{sec:KD}

For $p\ge 5$ and $m\ge 1$ define $M_{m,p}:=[q^{mp}]C_{\mathrm{mix}}(q)H_{\mathrm{mix}}(q^p)^m$.

\begin{proposition}[Main Frobenius term]\label{prop:mainfrobenius}
If $\chi_3(p)=1$, then $M_{m,p}\equiv A_m^{\mathrm{mix}}\pmod{p^4}$ for every $m\ge 1$.
\end{proposition}

\begin{proof}
Write $H_{\mathrm{mix}}^m=\sum h_j^{(m)}q^j$. Then $M_{m,p}=\sum_{j=0}^m c_{(m-j)p}^{\mathrm{mix}}h_j^{(m)}$. The split-prime Eisenstein tower (\Cref{thm:split-tower}) gives $c_{kp}^{\mathrm{mix}}\equiv c_k^{\mathrm{mix}}\pmod{p^4}$ for $k\ge 0$, hence $M_{m,p}\equiv [q^m]C_{\mathrm{mix}}H_{\mathrm{mix}}^m=A_m^{\mathrm{mix}}\pmod{p^4}$.
\end{proof}

\begin{definition}\label{def:Up}
$U_p(q):=\log(t(q)^p/t(q^p))$. Also $V_p(q):=\log(u(q)^p/u(q^p))$, $W_p(q):=\log(g(q)^p/g(q^p))$ where $g=1+27u$. Then $U_p=V_p-2W_p$.
\end{definition}

\begin{lemma}[Exponential layers]\label{lem:exp-layers}
Let $p\ge 7$. Then $U_p\in pq\Z_{(p)}[[q]]$ and the following hold.

\emph{(i) Tail bound for all $k\ge 1$:} for every integer $k\ge 1$,
\[
v_p\!\left(\frac{U_p^k}{k!}\right)\ge k-v_p(k!)\ge k-\frac{k-1}{p-1}.
\]
In particular for $p\ge 7$ and $k\ge 4$, $v_p(U_p^k/k!)\ge 4$. Equivalently, $e^{-XU_p}\bmod X^4$ has integral coefficients, and the omitted tail $\sum_{k\ge 4}(-X)^kU_p^k/k!$ lies in $p^4\,\Z_{(p)}[[q]][[X]]$.

\emph{(ii) Truncated identity:} for the formal variable $X$, modulo $p^4$,
\[
e^{-XU_p}\equiv 1+\sum_{\ell=1}^3\frac{(-X)^\ell}{\ell!}U_p^\ell\pmod{(p^4,X^4)}\quad\text{and}\quad e^{XU_p}\equiv 1+\sum_{\ell=1}^3\frac{X^\ell}{\ell!}U_p^\ell\pmod{(p^4,X^4)}.
\]

\emph{(iii) Frobenius identity for $H_{\mathrm{mix}}$:} in $\Q_p[[q]][[X]]$,
\[
H_{\mathrm{mix}}(q)^{pX}=H_{\mathrm{mix}}(q^p)^X\cdot e^{-XU_p(q)},
\]
where both sides are interpreted as formal exponentials.

\emph{(iv) Specialization at $X=m\in\Z_{\ge 1}$:} both sides of (iii) are honest power series in $q$ with $\Z_{(p)}$-coefficients, and modulo $p^4$,
\[
H_{\mathrm{mix}}(q)^{mp}\equiv H_{\mathrm{mix}}(q^p)^m\sum_{\ell=0}^3\frac{(-m)^\ell}{\ell!}U_p^\ell\pmod{p^4}.
\]
\end{lemma}

\begin{proof}
Coefficient-wise Frobenius gives $t(q)^p\equiv t(q^p)\pmod p$, hence $t(q)^p/t(q^p)\in 1+pq\Z_{(p)}[[q]]$ and its $\log$ lies in $pq\Z_{(p)}[[q]]$, proving the first claim.

For (i): from $U_p\in pq\Z_{(p)}[[q]]$, $U_p^k\in p^kq^k\Z_{(p)}[[q]]$. Legendre's formula gives $v_p(k!)=\sum_{i\ge 1}\lfloor k/p^i\rfloor\le k/(p-1)$ (with strict inequality $\le (k-1)/(p-1)$ when $k\ge 1$). Hence $v_p(U_p^k/k!)\ge k-(k-1)/(p-1)$. For $p\ge 7$ and $k\ge 4$: at $k=4$, $v_p(4!)=v_p(24)=0$ for $p\ge 5$, so $v_p(U_p^4/4!)\ge 4$. For $k\ge 5$ and $p\ge 7$, $(k-1)/(p-1)\le (k-1)/6$, so $v_p(U_p^k/k!)\ge k-(k-1)/6=(5k+1)/6\ge 26/6>4$. The integrality and divisibility of the omitted tail follows.

For (ii): in $\Q_p[[q]][[X]]$, $e^{-XU_p}=\sum_{k\ge 0}(-X)^kU_p^k/k!$; truncating to $k\le 3$ and using (i) gives the displayed congruences. The $e^{XU_p}$ version is identical.

For (iii): by definition of $U_p$, $t(q)^p/t(q^p)=e^{U_p(q)}$, hence $H_{\mathrm{mix}}(q^p)=q^p/t(q^p)$ and $H_{\mathrm{mix}}(q)^p=q^p/t(q)^p$ satisfy $H_{\mathrm{mix}}(q^p)/H_{\mathrm{mix}}(q)^p=t(q)^p/t(q^p)=e^{U_p(q)}$. Therefore $H_{\mathrm{mix}}(q^p)=H_{\mathrm{mix}}(q)^p\cdot e^{U_p(q)}$, and raising to the $X$-th power gives $H_{\mathrm{mix}}(q^p)^X=H_{\mathrm{mix}}(q)^{pX}\cdot e^{XU_p(q)}$, equivalent to the displayed identity.

For (iv): specialize $X=m\in\Z_{\ge 1}$ in (iii) and apply (ii) to the resulting $e^{-mU_p}$ factor; the omitted $k\ge 4$ tail is $p^4$-divisible.
\end{proof}

\begin{lemma}[$\sigma_p$-linearity of $\Lambda_p$]\label{lem:sigma-p-linearity}
For Laurent series $a,b\in\Z_{(p)}((q))$ with finite principal parts (equivalently, bounded below in $q$-order),
\[
\Lambda_p\bigl(\sigma_p(a)\cdot b\bigr)=a\cdot\Lambda_p(b).
\]
\end{lemma}

\begin{proof}
Write $b=\sum b_nq^n$ and $a=\sum a_nq^n$, so $\sigma_p(a)=\sum a_nq^{pn}$ and $\sigma_p(a)\cdot b=\sum_{n,k}a_nb_kq^{pn+k}$. The $q^{ps}$-coefficient is $\sum_{n}a_nb_{p(s-n)}=\sum_n a_n\cdot[q^{p(s-n)}]b$, so $\Lambda_p(\sigma_p(a)\cdot b)=\sum_s\bigl(\sum_n a_n\cdot[q^{p(s-n)}]b\bigr)q^s=\sum_s\sum_n a_n[q^{s-n}]\Lambda_p(b)\,q^s=a\cdot\Lambda_p(b)$, where the inner equation uses $[q^{p(s-n)}]b=[q^{s-n}]\Lambda_p(b)$ by definition of $\Lambda_p$.
\end{proof}

\begin{definition}[Frobenius defect and additive defect]\label{def:nu-phi}
Set
\[
\nu_p(q):=\frac{t(q^p)}{t(q)^p}=e^{-U_p(q)},\qquad
\phi_p(q):=\nu_p(q)-1.
\]
By \Cref{lem:exp-layers}, $U_p\in pq\Z_{(p)}[[q]]$, hence $\phi_p\in pq\Z_{(p)}[[q]]$.
\end{definition}

\begin{lemma}[Triangular conversion between $\phi_p$- and $U_p$-layers]\label{lem:phi-U-triangular}
Let $p\ge 7$. Set
\[
\Phi_\ell:=\frac{\phi_p^\ell}{p^\ell},\qquad
\mathcal U_\ell:=\frac{U_p^\ell}{p^\ell}\qquad(\ell=1,2,3).
\]
Then, with row $\ell$ understood modulo $p^{4-\ell}$,
\[
\begin{pmatrix}\Phi_1\\ \Phi_2\\ \Phi_3\end{pmatrix}
\equiv
\begin{pmatrix}-1 & p/2 & -p^2/6\\ 0 & 1 & -p\\ 0 & 0 & -1\end{pmatrix}
\begin{pmatrix}\mathcal U_1\\ \mathcal U_2\\ \mathcal U_3\end{pmatrix}.
\]
The transition matrix is triangular with unit diagonal entries $-1,1,-1$, hence invertible over $\Z_{(p)}$ in the filtered sense. Consequently, for any $\Z_p$-linear maps $\mathcal L_1,\mathcal L_2$ on $q$-series and any single scalar $\lambda\in\Z_p$, the vector congruences
\[
\mathcal L_1(\Phi_\ell)\equiv\lambda\,\mathcal L_2(\Phi_\ell)\pmod{p^{4-\ell}}\qquad(\ell=1,2,3)
\]
are equivalent to
\[
\mathcal L_1(\mathcal U_\ell)\equiv\lambda\,\mathcal L_2(\mathcal U_\ell)\pmod{p^{4-\ell}}\qquad(\ell=1,2,3).
\]
The same triangular conversion also transports membership in any fixed $\Z_p$-submodule of the target.
\end{lemma}

\begin{proof}
Since $U_p\in pq\Z_{(p)}[[q]]$, every term $U_p^k/(k!\,p^\ell)$ with $k\ge 4$ has $p$-valuation
\[
v_p\!\left(\frac{U_p^k}{k!\,p^\ell}\right)\ge k-v_p(k!)-\ell.
\]
For $k<p$ we have $v_p(k!)=0$, so the bound becomes $k-\ell\ge 4-\ell$. For $k\ge p$, Legendre's bound $v_p(k!)\le k/(p-1)\le k/6$ gives
\[
k-v_p(k!)\ge 5k/6\ge 4
\]
for $p\ge 7$ and $k\ge 5$; the remaining case $k=4$ has $k<p$. In every case $v_p(U_p^k/(k!\,p^\ell))\ge 4-\ell$. Hence
\[
\phi_p=e^{-U_p}-1\equiv -U_p+\frac{U_p^2}{2}-\frac{U_p^3}{6}\pmod{U_p^4},
\]
\[
\phi_p^2\equiv U_p^2-U_p^3\pmod{U_p^4},\qquad
\phi_p^3\equiv -U_p^3\pmod{U_p^4}.
\]
Dividing the $\ell$-th identity by $p^\ell$ and reading mod $p^{4-\ell}$ gives the displayed matrix. Diagonal entries $-1,1,-1$ are units in $\Z_{(p)}$, so the matrix is invertible.

For the proportionality assertion, apply the displayed triangular matrix to the two vectors
\[
\bigl(\mathcal L_i(\mathcal U_1),\mathcal L_i(\mathcal U_2),\mathcal L_i(\mathcal U_3)\bigr)^t,\qquad i=1,2.
\]
By $\Z_p$-linearity of $\mathcal L_i$, this gives the corresponding $\Phi_\ell$-vector. Hence the residual vector $\mathcal L_1(\Phi_\bullet)-\lambda\mathcal L_2(\Phi_\bullet)$ equals the same triangular matrix applied to $\mathcal L_1(\mathcal U_\bullet)-\lambda\mathcal L_2(\mathcal U_\bullet)$. Since the matrix is invertible with unit diagonal, vanishing in the rowwise moduli $p^{4-\ell}$ is equivalent in the two bases.
\end{proof}

\begin{lemma}[Layer--defect equivalence]\label{lem:layer-defect}
For split $p\ge 5$, set $F_r(q):=\Lambda_p(C_{\mathrm{mix}}/t^{rp})-C_{\mathrm{mix}}/t^r$, $r=1,2,3$. The conditions $F_r\equiv 0\pmod{p^4}$ for $r=1,2,3$ are equivalent to
\[
\Lambda_p(C_{\mathrm{mix}}U_p^\ell)\equiv 0\pmod{p^4},\qquad \ell=1,2,3.
\]
\end{lemma}

\begin{proof}
$t(q^p)^r/t(q)^{rp}=e^{-rU_p}$ and $U_p^4\equiv 0\pmod{p^4}$, so $F_r\equiv 0\pmod{p^4}$ becomes
\[
\sum_{\ell=1}^3\frac{(-r)^\ell}{\ell!}\Lambda_p(C_{\mathrm{mix}}U_p^\ell)\equiv 0\pmod{p^4},\qquad r=1,2,3,
\]
after using $\Lambda_p(C_{\mathrm{mix}})\equiv C_{\mathrm{mix}}\pmod{p^4}$ from the split-prime Eisenstein tower. The matrix $\bigl((-r)^\ell/\ell!\bigr)_{1\le r,\ell\le 3}$ has determinant equal to $1$, giving the equivalence.
\end{proof}

\begin{theorem}[Katz--Dwork reduction]\label{thm:KD-input}
For split $p\ge 7$, suppose that
\[
\Lambda_p(C_{\mathrm{mix}}U_p^\ell)\equiv 0\pmod{p^4},\qquad \ell=1,2,3.
\]
Then, for every $m\ge 1$,
\[
\Lambda_p(C_{\mathrm{mix}}H_{\mathrm{mix}}^{pm})\equiv C_{\mathrm{mix}}H_{\mathrm{mix}}^m\pmod{p^4}.
\]
Consequently,
\[
A_{mp}^{\mathrm{mix}}\equiv A_m^{\mathrm{mix}}\pmod{p^4}\qquad(m\ge 1).
\]
\end{theorem}

\begin{proof}
Fix $m\in\Z_{\ge 1}$. By \Cref{lem:exp-layers}(iii) specialized at $X=m$,
\[
H_{\mathrm{mix}}(q)^{pm}=H_{\mathrm{mix}}(q^p)^m e^{-mU_p(q)}.
\]
Since $H_{\mathrm{mix}}(q^p)^m=\sigma_p(H_{\mathrm{mix}}(q)^m)$, \Cref{lem:sigma-p-linearity} gives
\[
\Lambda_p(C_{\mathrm{mix}}H_{\mathrm{mix}}^{pm})
=H_{\mathrm{mix}}^m\Lambda_p(C_{\mathrm{mix}}e^{-mU_p}).
\]
By \Cref{lem:exp-layers}(i),
\[
e^{-mU_p}\equiv 1+\sum_{\ell=1}^3\frac{(-m)^\ell}{\ell!}U_p^\ell\pmod{p^4},
\]
because the omitted tail has coefficients in $p^4\Z_{(p)}[[q]]$ and $m\in\Z$. Applying $\Lambda_p$ and using the hypotheses together with
$\Lambda_p(C_{\mathrm{mix}})\equiv C_{\mathrm{mix}}\pmod{p^4}$ from the split-prime Eisenstein tower \Cref{thm:split-tower}, we obtain
\[
\Lambda_p(C_{\mathrm{mix}}e^{-mU_p})\equiv C_{\mathrm{mix}}\pmod{p^4}.
\]
Multiplication by the integral series $H_{\mathrm{mix}}^m\in\Z_{(p)}[[q]]$ gives the scalar Katz--Dwork congruence.

Extracting the coefficient of $q^m$ from both sides gives
\[
[q^m]\Lambda_p(C_{\mathrm{mix}}H_{\mathrm{mix}}^{pm})
=[q^{mp}]C_{\mathrm{mix}}H_{\mathrm{mix}}^{pm}=A_{mp}^{\mathrm{mix}},
\]
while
\[
[q^m]C_{\mathrm{mix}}H_{\mathrm{mix}}^m=A_m^{\mathrm{mix}}
\]
by \Cref{thm:LB}. This proves the coefficient congruence.
\end{proof}

\section{Witt--Cartier closure}\label{sec:WC}

\subsection{Stack-local descent at \texorpdfstring{$P_-$}{P-}}\label{subsec:stack-local-descent}\label{subsec:auxiliary-descent}

Throughout this section $p\ge 7$ is a split prime, $p\equiv 1\pmod 3$. Let $\mathcal X_0(3)$ denote the effective rigidification of the modular stack of generalized elliptic curves with $\Gamma_0(3)$-level structure, obtained from the Deligne--Mumford stack of \cite{DR} by rigidifying the generic central subgroup $\{\pm 1\}$. The elliptic point $P_-$ over $j=0$ has effective stabilizer $\mu_3$. On the unrigidified stack one obtains the same effective formal neighborhood together with a central $\mu_2$-gerbe; this generic $\mu_2$ acts trivially on the weight-$5$, $\chi_3$ line bundle since $\chi_3(-1)(-1)^5=1$. We work after choosing a primitive cube root of unity $\zeta_3$; for split primes $p\equiv 1\pmod 3$ this is already defined over $\Z_p$.

We use $x=27u$, $g=1+x$, $t=u/g^2$. The Fricke involution is $w_3(x)=x^{-1}$, $w_3(u)=1/(729u)$.

\begin{lemma}[Theta representation of $C_0$]\label{lem:C0-theta-representation}
One has
\[
C_0=\mathcal A^2\mathcal B^3.
\]
\end{lemma}

\begin{proof}
This is \Cref{lem:prelim-theta}.
\end{proof}

\begin{lemma}[Unit leading coefficients at $P_-$]\label{lem:Pminus-unit-leading}\label{lem:stack-P-minus}
Let $R=\Z_p$ with $p\ge 7$ split, and work after adjoining $\zeta_3$ if necessary. In the quotient-stack chart
\[
\widehat{\mathcal X}_{0(3),P_-}\simeq \bigl[\operatorname{Spf}R[[r]]/\mu_3\bigr],
\qquad \zeta_3\cdot r=\zeta_3 r,
\]
one may choose $r$ so that
\[
g=1+27u=r^3.
\]
For a $\mu_3$-equivariant local generator $e$ of $\omega^5(\chi_3)$, one has
\[
C_0=r^2c(r^3)e,\qquad uC_0=r^2c_u(r^3)e,
\]
with
\[
c(0),\,c_u(0)\in R^\times.
\]
Consequently $\ord^{\rm st}_{P_-}(g)=3$, $\ord^{\rm st}_{P_-}(C_0)=2$, $\ord^{\rm st}_{P_-}(uC_0)=2$, and the reductions of these local descriptions modulo $p^N$ remain valid for every $N\ge 1$.
\end{lemma}

\begin{proof}
The function $g=1+27u$ cuts out $P_-$ on the coarse curve $\mathbb P^1_u$, and $dg/du=27\in R^\times$, so $g$ is a coarse local parameter at $P_-$. Since the completed stack is the tame quotient with invariant coordinate $r^3$, after multiplying $r$ by a unit whose cube is the old leading unit we may arrange $g=r^3$; this is allowed because $3\in R^\times$.

We derive the local form $C_0=r^2c(r^3)e$ with $c(0)\in R^\times$ from the Borwein theta identity $\mathcal A^3=\mathcal B^3+\mathcal D^3$ rather than postulating the stabilizer character. From \Cref{thm:modular}(ii), $F_{\mathrm{mix}}=\mathcal A^3=(1+27u)\eta(\tau)^9/\eta(3\tau)^3=g\cdot F_0$, where $F_0=\eta(\tau)^9/\eta(3\tau)^3$ has unit leading coefficient $1$ at $P_-$ (since $\eta(\tau)$ and $\eta(3\tau)$ are nonvanishing on $\Gamma_0(3)$ away from cusps). Therefore $\mathcal A^3=g\cdot F_0=r^3\cdot F_0(0)\cdot(1+O(r))$, so $\mathcal A$ has order exactly $1$ in $r$ at $P_-$ with unit leading coefficient: $\mathcal A=r\cdot a_0(r^3)$ with $a_0(0)\in R^\times$. From $\mathcal D^3/\mathcal B^3=27u$, at $P_-$ this becomes $\mathcal D^3/\mathcal B^3=-1$, so $\mathcal D=-\zeta_3^j\mathcal B$ for some $j$ (in particular $\mathcal D$ is a unit times $\mathcal B$ at $P_-$). Substituting into $\mathcal A^3=\mathcal B^3+\mathcal D^3=\mathcal B^3(1+\mathcal D^3/\mathcal B^3)=\mathcal B^3(1+27u)=\mathcal B^3\cdot g$ gives $\mathcal B^3=\mathcal A^3/g=F_0$, so $\mathcal B^3$ is a unit at $P_-$, hence $\mathcal B$ is itself a unit at $P_-$. By \Cref{lem:C0-theta-representation}, $C_0=\mathcal A^2\mathcal B^3$. Therefore
\[
C_0=\mathcal A^2\cdot\mathcal B^3=r^2a_0(r^3)^2\cdot\mathcal B^3=r^2\cdot c(r^3)\cdot e
\]
where $c(r^3):=a_0(r^3)^2\mathcal B^3$ has unit leading coefficient at $r=0$ (because both $a_0(0)$ and $\mathcal B(0)$ are units), and $e$ is the local generator of $\omega^5(\chi_3)$ rendered $\mu_3$-equivariant by the choice of trivialization. The vanishing order $\ord^{\rm st}_{P_-}(C_0)=2$ follows directly, with the stabilizer character of $r^2$.

Since $u$ is a unit at $P_-$ (with $u(P_-)=-1/27$),
\[
uC_0=r^2\cdot u\,c(r^3)\cdot e,
\]
and $c_u(0):=u(P_-)\cdot c(0)=-c(0)/27\in R^\times$. Units in $R[[s]]$ remain units after reduction modulo $p^N$.
\end{proof}

\begin{lemma}[Quotient-stack chart at $P_-$]\label{lem:quotient-stack-chart}
After choosing $\zeta_3$ and an eigenparameter $r$ for the effective $\mu_3$-stabilizer, there is an isomorphism of completed formal stacks
\[
\widehat{\mathcal X}_{0(3),P_-}\,\widehat\otimes_{\Z_p}\Z_p[\zeta_3]\;\simeq\;\bigl[\operatorname{Spf}\Z_p[\zeta_3][[r]]\,\big/\,\mu_3\bigr],
\]
with $\zeta_3\cdot r=\zeta_3 r$. The invariant coarse coordinate is $s=r^3$.
\end{lemma}

\begin{proof}
The stack $\mathcal X_0(3)$ is tame at $P_-$ in the sense of \cite[Theorem~3.2]{AOV}: the effective stabilizer $\mu_3$ is a finite flat linearly reductive group scheme over $\Z_p$ for $p\ne 3$. By \cite[Theorem~3.2]{AOV}, étale-locally on the coarse moduli space at $P_-$, the stack admits a presentation as $[U/\mu_3]$. The stabilizer action on the completed tangent line at $P_-$ is effective and one-dimensional, hence its character is primitive; after choosing an eigenparameter $r$, the action is $r\mapsto\zeta_3 r$. The invariant coordinate is $s=r^3$.
\end{proof}

\begin{lemma}[Stabilizer functoriality of the canonical subgroup]\label{lem:mu3-preserves-canonical-subgroup}
Let $p\ge 7$ be split, $p\equiv 1\pmod 3$, and let $S$ be a $p$-adically complete $\Z_p$-scheme. Let $(E,C_3)/S$ be an ordinary generalized elliptic curve with $\Gamma_0(3)$-structure, and let $C_p(E)\subset E[p]$ be Katz's canonical subgroup. If $\alpha:(E,C_3)\xrightarrow{\sim}(E',C_3')$ is an isomorphism of $\Gamma_0(3)$-objects, then $\alpha(C_p(E))=C_p(E')$. Consequently, if $\alpha$ is an automorphism of $(E,C_3)$, then $\alpha(C_p(E))=C_p(E)$ and $\alpha$ descends to an automorphism $\bar\alpha:E/C_p(E)\xrightarrow{\sim}E/C_p(E)$ for which the quotient isogeny $\pi:E\to E/C_p(E)$ satisfies $\pi\circ\alpha=\bar\alpha\circ\pi$. In particular, at the $j=0$ point $P_-$ of $\mathcal X_0(3)$, the effective stabilizer $\mu_3$ preserves the canonical subgroup.
\end{lemma}

\begin{proof}
Katz's canonical subgroup on the ordinary locus is functorial under isomorphisms of elliptic curves \cite[\S 3.10]{Katz-padic}. For the present use only the automorphism case is needed: if $\alpha:E\to E'$ is an isomorphism, then $\alpha(C_p(E))\subset E'[p]$ has the same defining canonical-subgroup property as $C_p(E')$, and by uniqueness in Katz's construction $\alpha(C_p(E))=C_p(E')$. If $E'=E$ and $\alpha$ is an automorphism of $(E,C_3)$, this gives $\alpha(C_p(E))=C_p(E)$, so $\alpha$ induces a unique automorphism $\bar\alpha$ of $E/C_p(E)$ satisfying the displayed commutative diagram. At $P_-$, the effective stabilizer on the rigidified stack is generated by the order-$3$ automorphism $[\zeta_3]$ of the $j=0$ elliptic curve which preserves the chosen cyclic subgroup $C_3$. Since $p\equiv 1\pmod 3$, the $j=0$ curve is ordinary, so $C_p$ is defined there. Applying the preceding paragraph to $\alpha=[\zeta_3]$ gives $[\zeta_3](C_p)=C_p$.
\end{proof}

\begin{lemma}[Descent of Katz Frobenius and equivariance at $P_-$]\label{lem:katz-trace-descent}\label{lem:canonical-frobenius-via-atlas}\label{prop:descent-to-X03}\label{prop:local-frobenius-normal-form}\label{prop:stack-local-form-P-minus}
Fix $M=4$ and put $Y:=\mathcal X_0(3)\times_{\mathcal X(1)}Y_1(4)$. Then $Y$ is a fine modular scheme over $\Z_p$, and $\pi:Y\to\mathcal X_0(3)$ is a representable finite étale atlas over the ordinary locus. Katz's canonical Frobenius lift and normalized weight-$k$ trace operator on $Y^{\rm ord}$ descend to $\mathcal X_0(3)^{\rm ord}$. Moreover, after choosing the quotient-stack chart
\[
\widehat{\mathcal X}_{0(3),P_-}\widehat\otimes_{\Z_p}\Z_p[\zeta_3]\simeq \bigl[\operatorname{Spf}\Z_p[\zeta_3][[r]]/\mu_3\bigr],\qquad \zeta_3\cdot r=\zeta_3 r,
\]
the descended canonical Frobenius lift $\Phi$ is $\mu_3$-equivariant. Equivalently, if $F(r):=\Phi^*r\in\Z_p[\zeta_3][[r]]$, then $F(\zeta_3 r)=\zeta_3 F(r)$, so
\[
F(r)=r\,G(r^3),\qquad G(s)\in\Z_p[\zeta_3][[s]].
\]
Since $\Phi$ reduces to absolute Frobenius modulo $p$, $F(r)\equiv r^p\pmod p$. For split primes $p\equiv 1\pmod 3$,
\begin{equation}\label{eq:Phi-local}
\Phi^*r=r^p+pr\alpha_1(r^3)+p^2r\alpha_2(r^3)+p^3r\alpha_3(r^3)+O(p^4),\qquad \alpha_i\in\Z_p[\zeta_3][[s]].
\end{equation}
\end{lemma}

\begin{proof}
The auxiliary $\Gamma_1(4)$-structure makes $Y$ fine. We verify this by checking that the only automorphism of a generalized elliptic curve $E/\Z_p$ (with $p\ne 2,3$) that fixes a $\Gamma_1(4)$-structure $\beta:\Z/4\hookrightarrow E$ is the identity. The automorphism group of a generalized elliptic curve is at most $\mu_6=\{\pm 1\}\times\mu_3$ (at $j=0$) or $\{\pm 1,\pm i\}=\mu_4$ (at $j=1728$), reducing to $\{\pm 1\}$ for generic $j$, by the standard classification \cite[\S 2.7]{KM85}. We handle each case:

\emph{Case $[-1]$ (always present).} The automorphism $[-1]$ acts on $E[4]$ as $-\mathrm{id}$ and therefore inverts a $\Gamma_1(4)$-section $\beta(1)$ of exact order $4$; it can fix $\beta$ only if $\beta(1)=-\beta(1)$, i.e., $\beta(1)\in E[2]$, contradicting exact order $4$. So $[-1]$ does not fix any $\Gamma_1(4)$-structure.

\emph{Case $[\zeta_3]$ at $j=0$.} On the $j=0$ elliptic curve the automorphism $[\zeta_3]$ is the standard $\Z[\zeta_3]$-multiplication. In the $\Z[\zeta_3]$-module $E[4]$ (well-defined since $\gcd(3,4)=1$), the endomorphism $\zeta_3-1\in\Z[\zeta_3]$ has norm $|1-\zeta_3|^2=3$. Since $\gcd(3,4)=1$, multiplication by $\zeta_3-1$ is invertible on $E[4]$. Hence $(\zeta_3-1)\beta(1)=0$ forces $\beta(1)=0$, contradicting that $\beta(1)$ has exact order $4$. Therefore no nontrivial $\zeta_3$-action fixes $\beta$.

\emph{Case $[i]$ at $j=1728$.} The automorphism $[i]$ acts on the tangent space as multiplication by $i\in\Z_p$ (after possibly adjoining $i$; for the moduli interpretation we may work over $\Z_p[i]$). Hence $[i]$ acts on $E[4]$ as multiplication by $i$. For $\beta(1)\in E[4]$ of exact order $4$, $[i]\beta(1)=i\cdot\beta(1)$, which equals $\beta(1)$ only if $(i-1)\beta(1)=0$. Since $i-1$ has norm $2$ in $\Z[i]$, and $\gcd(2,4)=2\ne 1$, this requires $\beta(1)$ to lie in the $2$-torsion subgroup $E[2]$, contradicting exact order $4$. The squared automorphism $[i]^2=[-1]$ is handled by the first case, and $[i]^3=[-i]$ is treated symmetrically. Hence no nontrivial element of $\mu_4$ fixes $\beta$.

Thus the only automorphism fixing a $\Gamma_1(4)$-structure is the identity, and $Y$ is fine. Since $(4,3p)=1$, the map $Y\to\mathcal X_0(3)$ is representable and finite étale over the ordinary locus. On $Y^{\rm ord}$, Katz's construction sends
\[
(E,C_3,\beta)\longmapsto(E/C_p,\bar C_3,\bar\beta),
\]
where $\beta$ is the $\Gamma_1(4)$-structure and $\bar C_3$, $\bar\beta$ are the images under the quotient isogeny $E\to E/C_p$. The subgroup $C_3$ and the $\Gamma_1(4)$-structure transport isomorphically because their orders are prime to $p$.

Pull this construction back to $Y\times_{\mathcal X_0(3)}Y$. An $S$-point of this fibre product is one elliptic curve with one $\Gamma_0(3)$-subgroup and two auxiliary $\Gamma_1(4)$-rigidifications; the two projections to $Y$ differ only by which rigidification is remembered. On the fine atlas $Y$, the two pullbacks of the canonical Frobenius lift along the two projections $Y\times_{\mathcal X_0(3)}Y\rightrightarrows Y$ agree because the canonical subgroup is functorial under isomorphisms of generalized elliptic curves (Katz \cite[Theorem~3.10.7]{Katz-padic}). The normalized weight-$k$ trace operator on $\omega^k$ is similarly compatible with these isomorphisms, since the quotient isogeny, the induced map on the Hodge bundle, and the sheaf trace are natural in the descent groupoid. Therefore Katz's construction gives a descent datum in the stack groupoid, with the cocycle on the triple fibre product inherited from functoriality of composition, and the Frobenius lift and normalized trace descend from $Y^{\rm ord}$ to $\mathcal X_0(3)^{\rm ord}$.

For the stabilizer equivariance at $P_-$: choose a point of the atlas above $P_-$, represented by $(E_0,C_3,\beta)$. The order-$3$ stabilizer element $[\zeta_3]$ does not fix $\beta$; rather it gives an arrow in the descent groupoid $(E_0,C_3,\beta)\to(E_0,C_3,[\zeta_3]\beta)$. By \Cref{lem:mu3-preserves-canonical-subgroup}, $[\zeta_3](C_p)=C_p$, so after quotienting by $C_p$ the same arrow descends to $(E_0/C_p,\bar C_3,\bar\beta)\to(E_0/C_p,\bar C_3,[\zeta_3]\bar\beta)$. Thus the descended Frobenius morphism commutes with the $\mu_3$-action in the quotient-stack chart. In coordinates, $\zeta_3^*\Phi^*r=\Phi^*\zeta_3^*r$, and since $\zeta_3^*r=\zeta_3 r$, this becomes $F(\zeta_3 r)=\zeta_3 F(r)$. Therefore $F(r)/r$ is $\mu_3$-invariant and lies in $\Z_p[\zeta_3][[r^3]]$. Finally, the canonical Frobenius lift reduces modulo $p$ to absolute Frobenius on the special fibre, so $F(r)\equiv r^p\pmod p$. Because $p\equiv 1\pmod 3$, $r^p=r(r^3)^{(p-1)/3}$. Writing
\[
G(s)=s^{(p-1)/3}+p\alpha_1(s)+p^2\alpha_2(s)+p^3\alpha_3(s)+O(p^4)
\]
gives \eqref{eq:Phi-local}.
\end{proof}

\begin{remark}
$M=4$ is a tool, not a property of the result. Any $M\ge 3$ with $(M,3p)=1$ killing $\{\pm 1\}$ produces the same descended operator: on $Y\times_{\mathcal X_0(3)}Y'$ for two such atlases $Y$, $Y'$, both lifts of the canonical Frobenius send $(E,C_3,\beta,\beta')$ to $(E/C_p,\bar C_3,\bar\beta,\bar\beta')$, since the canonical subgroup is functorial in $E$ alone and prime-to-$3p$ level structures transport isomorphically.
\end{remark}

\begin{remark}\label{rem:P-minus-ordinary}
The elliptic curve $E_0$ with $j=0$ has CM by $\Z[\zeta_3]$, and by Deuring's criterion has ordinary reduction at $p$ if and only if $p$ splits in $\Z[\zeta_3]$, i.e., $p\equiv 1\pmod 3$. Under the split-prime hypothesis Katz's canonical subgroup is therefore defined at $P_-$, and the descent of \Cref{lem:katz-trace-descent} carries both the Frobenius lift $\Phi$ and the normalized weight-$5$ trace operator $\mathcal C_{\Phi,5}$ to $\mathcal X_0(3)^{\rm ord}$; the latter realizes $\Lambda_p$ on $q$-expansions by \Cref{lem:trace-operator-w5}.
\end{remark}

\begin{remark}[Inert primes]
For inert primes $p\equiv 2\pmod 3$, the Frobenius acts on the character group of $\mu_3$ by $\chi\mapsto\chi^p=\chi^{-1}$. The local Frobenius is therefore semilinear on the stabilizer character rather than preserving the weight-one character of $r$; the split-prime normal form $r\cdot G(r^3)$ is unavailable. This mirrors the inert obstruction proved in \Cref{sec:inert} via the $q$-side.
\end{remark}

\begin{lemma}[Descended Katz trace: local formula and $q$-expansion]\label{lem:trace-operator-w5}
Let $\Phi$ be the descended canonical Frobenius lift on $\mathcal X_0(3)^{\rm ord}$ from \Cref{lem:katz-trace-descent}. Let $A=\Z_p[\zeta_3][[r]]$ be the stack-local cover at $P_-$. Let $e,e'$ be $\mu_3$-equivariant local generators of the weight-$5$, $\chi_3$ line bundle near the source and target of $\Phi$, and write
\[
\Phi^*e'=p^5J_\Phi e,\qquad J_\Phi\in A^\times.
\]
Then $J_\Phi\in A^{\mu_3,\times}=\Z_p[\zeta_3][[r^3]]^\times$, and the descended normalized Katz trace is given locally by
\begin{equation}\label{eq:CPhi5-def}
\mathcal C_{\Phi,5}(fe)=\frac{1}{p}\Tr_{A/\Phi^*A}\!\bigl(J_\Phi^{-1}f\bigr)e'.
\end{equation}
At the cusp $\infty$, with Tate parameter $q$, the same descended operator acts on finite Laurent $q$-expansions by
\[
\mathcal C_{\Phi,5}\!\left(\sum_{n\gg-\infty}a_nq^n\right)=\sum_{n\gg-\infty}a_{pn}q^n=\Lambda_p\!\left(\sum_{n\gg-\infty}a_nq^n\right).
\]
Thus the stack-local formula at $P_-$ and the Cartier operator on $q$-expansions are restrictions of one descended global operator.
\end{lemma}

\begin{proof}
The local formula follows from the definition of Katz's normalized trace. Since $\Phi^*e'=p^5J_\Phi e$, $e=p^{-5}J_\Phi^{-1}\Phi^*e'$, so the sheaf trace on $\omega^{\otimes 5}(\chi_3)$ sends $fe=p^{-5}J_\Phi^{-1}f\,\Phi^*e'$ to $p^{-5}\Tr_{A/\Phi^*A}(J_\Phi^{-1}f)e'$. Katz's normalized weight-$5$ trace is $p^4$ times this sheaf trace \cite[\S 3.11]{Katz-padic}, hence the coefficient is $p^4p^{-5}=1/p$.

The invariant property of $J_\Phi$ is forced by equivariance. If $\lambda\in\mu_3$, then $\Phi$ is $\mu_3$-equivariant by \Cref{lem:katz-trace-descent}. The local generators $e,e'$ have the same $\mu_3$-character, because they generate the same automorphic line bundle on source and target. Applying $\lambda^*$ to $\Phi^*e'=p^5J_\Phi e$ gives $\Phi^*(\lambda^*e')=p^5\lambda^*(J_\Phi)\lambda^*e$. Since $\lambda^*e$ and $\lambda^*e'$ are multiplied by the same character, comparison gives $\lambda^*(J_\Phi)=J_\Phi$. Thus $J_\Phi\in A^{\mu_3}=\Z_p[\zeta_3][[r^3]]$. It is a unit because the canonical Frobenius is an isogeny of ordinary elliptic curves and the normalized pullback on the Hodge bundle is a unit after extracting the explicit factor $p$ in weight $1$.

Now restrict to the Tate chart at $\infty$. The universal object is $\operatorname{Tate}(q)$ with its $\Gamma_0(3)$-subgroup. The canonical subgroup is $\mu_p$, and the quotient is $\operatorname{Tate}(q)/\mu_p\simeq\operatorname{Tate}(q^p)$. For the invariant differential $\omega=dT/T$, the quotient map satisfies $\Phi^*\omega'=p\,\omega$. Thus in weight $5$ we have $J_\Phi=1$ on the Tate chart. Let $Q=q^p$. The extension of Laurent series rings is $\Z_p((Q))\subset\Z_p((q))$, with basis $1,q,\ldots,q^{p-1}$. Multiplication by $q^a$, $1\le a\le p-1$, has trace $0$ in this basis, while multiplication by $q^{pm}=Q^m$ has trace $pQ^m$. Hence
\[
\Tr_{\Z_p((q))/\Z_p((q^p))}(q^n)=\begin{cases}p\,q^{n/p},& p\mid n,\\ 0,& p\nmid n.\end{cases}
\]
The local formula therefore gives $\mathcal C_{\Phi,5}(\sum a_nq^n)=\frac{1}{p}\sum a_n\Tr(q^n)=\sum a_{pn}q^n=\Lambda_p$.
\end{proof}

\begin{remark}\label{rem:JPhi-suppression}
$J_\Phi^{-1}\in\Z_p[\zeta_3][[r^3]]^\times$ is $\mu_3$-invariant of $p$-valuation $0$, hence suppresses without affecting $\mu_3$-character, $p$-valuation, or $r$-exponent mod $3$ in the pole-counting argument below.
\end{remark}

\subsection{Length-three Witt--Cartier pole estimate}\label{subsec:pole-estimate}

We use throughout the local setup of \Cref{subsec:stack-local-descent}: $A=\Z_p[\zeta_3][[r]]$, $\zeta_3\cdot r=\zeta_3 r$, $\Phi^*r=F(r)$ as in \eqref{eq:Phi-local}, $\mathcal C_{\Phi,5}$ as in \Cref{lem:trace-operator-w5}, $C_0=r^2c(r^3)e$. By \Cref{rem:JPhi-suppression}, $J_\Phi^{-1}$ is suppressed.

Recall $\nu_p,\phi_p$ from \Cref{def:nu-phi}: $\nu_p=e^{-U_p}=t(q^p)/t(q)^p$ and $\phi_p=\nu_p-1\in pq\Z_{(p)}[[q]]$.

Since $F(r)\equiv r^p\pmod p$, $F$ is distinguished of Weierstrass degree $p$; Weierstrass preparation makes $A$ finite free of rank $p$ over $\Phi^*A=\Z_p[\zeta_3][[F(r)]]$. Write $R:=F(r)$ for the target parameter.

\begin{lemma}[Trace-residue formula on the tame $\mu_3$-chart]\label{lem:trace-residue-tame}
Let $A=\Z_p[\zeta_3][[r]]$, let $F(r)\in r\,A$ be a Frobenius lift on $A$, i.e.\ $F(r)\equiv r^p\pmod p$. Thus $F$ is distinguished of Weierstrass degree $p$; set $R=F(r)$, so $B:=\Z_p[\zeta_3][[R]]\subset A$ is a finite flat extension of rank $p$. Let $J_\Phi(r)\in A^\times$ denote the Jacobian of $\Phi$ with respect to the chosen weight-$5$ trivialization $e$, so that the normalized weight-$5$ Katz trace satisfies $\mathcal C_{\Phi,5}(he)=\bigl(p^{-1}\Tr_{A/B}(h\,J_\Phi^{-1})\bigr)\,e$ for $h$ in the field of fractions of $A$. Then for every $K\ge 1$ and every $h\in\operatorname{Frac}(A)$ with at worst a Laurent principal part at $r=0$,
\begin{equation}\label{eq:trace-residue}
[R^{-K}]\,\mathcal C_{\Phi,5}(he)=\frac{1}{p}\Res_{r=0}\!\bigl(h(r)\,J_\Phi(r)^{-1}\,F'(r)\,F(r)^{K-1}\,\dd r\bigr).
\end{equation}
In particular $J_\Phi^{-1}\in\Z_p[\zeta_3][[r^3]]^\times$ is a $\mu_3$-invariant unit on the tame chart, so it does not affect the $r$-valuation, the $p$-valuation, or the $\mu_3$-character class of the residue integrand.
\end{lemma}

\begin{proof}
For the principal part at $R=0$, the trace-residue formula for finite flat extensions of complete local rings (see \cite[\S 3]{Tate-residues} or \cite[Ch.~III]{Hartshorne-RD}) gives, for $K\ge 1$,
\[
[R^{-K}]\Tr_{A/B}(h)=[R^{-K}]\Res_{r=0}\!\left(\frac{h(r)F'(r)}{R-F(r)}\,\dd r\right).
\]
In the Laurent ring $A((R^{-1}))$,
\[
\frac{1}{R-F(r)}=\frac{1}{R}\cdot\frac{1}{1-F(r)/R}=\sum_{K\ge 1}F(r)^{K-1}R^{-K},
\]
viewed as a formal expansion in the independent target variable $R^{-1}$ in $A((R^{-1}))$; no ordinary $r$-order assertion on $F$ is used. The required finiteness and rank-$p$ trace are supplied by the Weierstrass degree-$p$ property $F\equiv r^p\pmod p$. Therefore
\[
[R^{-K}]\Tr_{A/B}(h)=\Res_{r=0}\!\bigl(h(r)F'(r)F(r)^{K-1}\,\dd r\bigr).
\]
Replacing $h$ by $hJ_\Phi^{-1}$ and dividing by $p$ gives \eqref{eq:trace-residue}.

The statement on $J_\Phi^{-1}$: by \Cref{lem:Pminus-unit-leading,lem:trace-operator-w5}, $J_\Phi$ is a unit on the tame $\mu_3$-chart and lies in the $\mu_3$-invariant subring $\Z_p[\zeta_3][[r^3]]$. As a unit in this subring it has $r$-valuation $0$ and $\mu_3$-character class $0$, and its constant term is a $p$-adic unit, so it has $p$-valuation $0$. Therefore multiplying the integrand by $J_\Phi^{-1}$ preserves all three quantities used in the pole estimate.
\end{proof}

\begin{remark}[Sign bookkeeping]\label{rem:trace-residue-sign}
If one writes the trace-residue formula with denominator $F(r)-R$, an overall minus sign must be inserted:
\[
[R^{-K}]\Tr_{A/B}(h)=-[R^{-K}]\Res_{r=0}\!\left(\frac{h(r)F'(r)}{F(r)-R}\,\dd r\right).
\]
Together with
\[
\frac{1}{F(r)-R}=-\sum_{K\ge 1}F(r)^{K-1}R^{-K},
\]
this gives the same positive coefficient formula \eqref{eq:trace-residue}. The pole estimate uses only valuations and $\mu_3$-character classes, but the displayed coefficient identity is fixed with the positive sign.
\end{remark}

\begin{proposition}[Length-three Witt--Cartier pole estimate]\label{prop:pole-estimate}\label{prop:pole-estimate-star}
For every $\ell\in\{1,2,3\}$ and every $s$ with $\ell\le s\le 3$,
\begin{equation}\label{eq:star}
p^{s-\ell}\,\mathcal C_{\Phi,5}\!\bigl(r^{2+s(1-p)}\Z_p[\zeta_3][[r^3]]\,e\bigr)\subset r^{-1}\Z_p[\zeta_3][[r^3]]\,e\pmod{p^{4-\ell}}.
\end{equation}
\end{proposition}

\begin{proof}
By \Cref{lem:trace-residue-tame}, the principal-part coefficient in $R$ is
\begin{equation}\label{eq:residue-formula}
[R^{-K}]\,\mathcal C_{\Phi,5}(fe)=\frac{1}{p}\Res_{r=0}\!\bigl(f(r)\,J_\Phi(r)^{-1}\,F'(r)\,F(r)^{K-1}\,\dd r\bigr),
\end{equation}
where $J_\Phi^{-1}\in\Z_p[\zeta_3][[r^3]]^\times$ has $r$-valuation $0$, $p$-valuation $0$, and $\mu_3$-character class $0$, hence does not affect the estimates below; we suppress it in the rest of the proof and refer to \Cref{lem:trace-residue-tame} for justification. After identifying the target completed disk with $\operatorname{Spf}A$ and relabelling $R$ back to $r$, this is the $r^{-K}$ coefficient.

\medskip\noindent\textbf{Step 1: $\mathbb Z/3$-counting forces $K\equiv 1\pmod 3$.}\quad
$f\in r^{2+s(1-p)}\Z_p[\zeta_3][[r^3]]$ contributes $r$-exponent $\equiv 2+s(1-p)\equiv 2\pmod 3$ since $1-p\equiv 0\pmod 3$. The derivative
\[
F'(r)=p\,r^{p-1}+pD_1(r)+p^2D_2(r)+p^3D_3(r),\quad D_i(r)=\alpha_i(r^3)+3r^3\alpha_i'(r^3),
\]
has both terms in exponent class $\equiv 0\pmod 3$. The factor $F(r)^{K-1}$ has each factor of $F$ in class $\equiv 1\pmod 3$, so contributes $\equiv K-1\pmod 3$. A nonzero residue requires total exponent $-1\equiv 2\pmod 3$, giving $2+0+(K-1)\equiv 2\pmod 3$, hence $K\equiv 1\pmod 3$. For $K\ge 2$ this forces $K\ge 4$.

\medskip\noindent\textbf{Step 2: Valuation count.}\quad
Write $F(r)^{K-1}=r^{p(K-1)}(1+X)^{K-1}$ with $X=pr^{1-p}\alpha_1+p^2r^{1-p}\alpha_2+p^3r^{1-p}\alpha_3$. The binomial expansion selects $b\le K-1$ factors of $X$, each contributing minimal $r$-exponent $1-p$ and $p$-valuation at least $1$.

\emph{Derivative-high case} ($F'\to pr^{p-1}$): total $r$-exponent
\[
E_{\rm high}=2+s(1-p)+(p-1)+p(K-1)+b(1-p)+3h=1+s+b+p(K-b-s)+3h
\]
with $h\ge 0$. Residue condition $E_{\rm high}=-1$:
\[
(p-1)(s+b)=pK+2+3h,\qquad s+b-K=\frac{K+2+3h}{p-1}>0,
\]
hence $s+b\ge K+1$, so $b\ge K+1-s$. With $K\ge 4$, $b\ge 5-s$.

\emph{Derivative-low case} ($F'\to p^jD_j$, $j\ge 1$): similar computation gives
\[
(p-1)(s+b)=p(K-1)+3+3h,\qquad s+b-(K-1)=\frac{K+2+3h}{p-1}>0,
\]
hence $s+b\ge K$ and $b\ge K-s\ge 4-s$ when $K\ge 4$.

\emph{Control example} ($p=7$, $K=4$, $s=1$, $h=0$): low case requires $6(s+b)=24$, so $s+b=4$ and $b=3$, satisfying $b\le K-1=3$. High case requires $6(s+b)=30$, so $b=4>K-1=3$, hence does not occur in this minimal example.

\medskip\noindent\textbf{Step 3: Total valuation.}\quad
Before the external $1/p$, each integrand has $p$-valuation $\ge 1+b$ (the explicit $p$ from $F'$ plus at least $b$ from the $X$-expansion). After $1/p$, valuation $\ge b\ge 4-s$ (the minimum, from the low case). Multiplying by external $p^{s-\ell}$ in \eqref{eq:star} gives valuation $\ge (4-s)+(s-\ell)=4-\ell$. Thus every $r^{-K}$ coefficient with $K\ge 2$ vanishes mod $p^{4-\ell}$.

The suppressed $O(p^4)$ part of $F$ contributes valuation $\ge 4$ before $1/p$, hence $\ge 3$ after, already vanishing for all $\ell\in\{1,2,3\}$.

\medskip\noindent\textbf{Step 4: $K=1$.}\quad
$F^{K-1}=1$, and the residue lands in the $r^{-1}$ part of the principal part. Combining steps gives \eqref{eq:star}.
\end{proof}

\begin{remark}[The role of the $\mathbb Z/3$-counting]
Without the $\mu_3$-equivariant congruence-class exclusion, the residue calculation would not rule out the lower principal parts $K=2,3$. The split CM argument gains its extra digit precisely by excluding those two possibilities and forcing $K\ge 4$; this is the local mechanism behind the improvement from the generic weight-$3$ expectation $p^3$ \cite{RRV} to the $p^4$ estimate proved here.
\end{remark}

\subsection{Hecke finite differences and Cartier--Hecke comparison}\label{subsec:hecke-finite-differences}

For $\ell=1,2,3$ define first the Hecke finite-difference numerator for $C_0$:
\begin{equation}\label{eq:hecke-fd-numerator}
N_\ell(C_0):=\sum_{j=0}^\ell(-1)^{\ell-j}\binom{\ell}{j}t^j\,T_p\!\left(\frac{C_0}{t^{jp}}\right),
\end{equation}
where $T_p$ is the Hecke operator on weakly holomorphic weight-$5$, $\chi_3$ forms. After the global divisibility statement of \Cref{lem:divisibility-Nell}, we write
\begin{equation}\label{eq:hecke-fd}
\mathfrak E_\ell:=N_\ell(C_0)/p^\ell.
\end{equation}

\begin{lemma}[Hecke action on weakly holomorphic forms]\label{lem:hecke-weakly-holomorphic}
Let $f$ be a weight-$5$, $\chi_3$ form on $\Gamma_0(3)$ over $\Z_{(p)}$, weakly holomorphic with finite principal parts at both cusps: $f|_\infty=\sum_{n\ge-m_\infty}a_nq^n$ and $f|_0=\sum_{n\ge-m_0}b_nq_0^n$. Then the Hecke operator $T_p$ (where $p\ne 3$ prime) defined via the double coset $\Gamma_0(3)\bigl(\begin{smallmatrix}1&0\\0&p\end{smallmatrix}\bigr)\Gamma_0(3)$ acting in weight $5$ with character $\chi_3$ \cite[\S 5]{AtkinLehner70}, \cite[\S 5.1]{DS} preserves the space of such forms and on the cusp-$\infty$ $q$-expansion is given by
\[
T_pf(q)=\sum_n a_{pn}q^n+\chi_3(p)p^4\sum_n a_nq^{pn}.
\]
At split primes $\chi_3(p)=1$ this reads $T_p f=\Lambda_p f+p^4\operatorname{Ver}_p(f)$ where $\operatorname{Ver}_p(f)(q)=f(q^p)$. The operator preserves $\Z_{(p)}$-integrality and commutes with reduction modulo $p^k$ for every $k\ge 1$.
\end{lemma}

\begin{proof}
Let $X=X_0(3)_{\Z_{(p)}}$ and let $\omega^5(\chi_3)$ denote the corresponding automorphic line bundle of weight $5$ and character $\chi_3$. The Hecke correspondence
\[
X\xleftarrow{\pi_1} X_0(3p)\xrightarrow{\pi_2} X
\]
with $(p,3)=1$ is a finite flat correspondence over $\Z_{(p)}$, étale on the generic fibre and on the étale (Verschiebung) component over the ordinary locus, away from the supersingular and cusp loci. It has finite degree on the compact curve. The pullback $\pi_1^*$ and the trace $(\pi_2)_*$ send weakly holomorphic sections of $\omega^5(\chi_3)$ to weakly holomorphic sections; this is because both maps are finite morphisms of compact curves and therefore preserve sections with poles supported on the cuspidal locus (the only divisorial poles of weakly holomorphic forms). Hence $T_p:=(\pi_2)_*\pi_1^*$ preserves the space of weakly holomorphic weight-$5$, $\chi_3$ forms over $\Z_{(p)}$.

To compute the $q$-expansion at the cusp $\infty$ on Laurent expansions, we use the standard double-coset description of the Hecke operator. For $p\ne 3$, by \cite[Lem.~5.1.1 and Prop.~5.2.1]{DS} (which describe the double-coset decomposition on $\Gamma_0(3)$ and the action of $\Gamma_0(3)\bigl(\begin{smallmatrix}1&0\\0&p\end{smallmatrix}\bigr)\Gamma_0(3)$ in weight $k$ with character $\chi$), the orbits split as
\[
\Gamma_0(3)\!\begin{pmatrix}1&0\\0&p\end{pmatrix}\!\Gamma_0(3)=\bigsqcup_{j=0}^{p-1}\Gamma_0(3)\begin{pmatrix}1&j\\0&p\end{pmatrix}\,\sqcup\,\Gamma_0(3)\begin{pmatrix}p&0\\0&1\end{pmatrix},
\]
each $j\in\{0,\dots,p-1\}$ contributing $p^{-1}\sum_jf((\tau+j)/p)$ and the last orbit contributing $p^{k-1}\chi(p)f(p\tau)$ in weight $k$. Apply this in weight $k=5$ and character $\chi_3$; expand $f|_\infty=\sum_n a_nq^n$ with $q=e^{2\pi i\tau}$. The sum $\sum_{j=0}^{p-1}e^{2\pi i n(\tau+j)/p}$ equals $p\cdot q^{n/p}$ when $p\mid n$ and $0$ otherwise, so the first set of orbits contributes $\sum_{p\mid n}a_n q^{n/p}=\sum_{n}a_{pn}q^n=\Lambda_p f(q)$. The last orbit contributes $p^{k-1}\chi_3(p)f(p\tau)=p^4\chi_3(p)\sum_n a_n q^{pn}=p^4\chi_3(p)\operatorname{Ver}_p f(q)$. Combining gives the displayed formula.

The same derivation applies on Laurent series with finite principal parts, since the orbit sums are linear in the input and the formulas $\Lambda_p$ and $\operatorname{Ver}_p$ extend coefficient-wise to Laurent series: if $f$ has principal part $\sum_{n=-m_\infty}^{-1}a_nq^n$, then $\Lambda_p f$ has principal part of length $\le \lceil m_\infty/p\rceil$, and $\operatorname{Ver}_p f$ has principal part of length $pm_\infty$. Hence $T_p$ sends weakly holomorphic forms to weakly holomorphic forms at the cusp $\infty$; by symmetric argument applied via the Fricke involution $w_3$, the same holds at the cusp $0$. (For an exposition of the weakly holomorphic version of the Hecke formalism applied to modular forms with finite principal parts, see \cite[\S 2]{Guerzhoy08}.)

Integrality over $\Z_{(p)}$ is immediate from the formula since $a_n\in\Z_{(p)}$ and the operations $\Lambda_p$, $\operatorname{Ver}_p$ preserve $\Z_{(p)}$-coefficients; commutation with reduction modulo $p^k$ follows since $T_p$ is a $\Z_{(p)}$-linear operator. At split primes $\chi_3(p)=1$.
\end{proof}

\begin{lemma}[Saturated weak $q$-expansion lattice]\label{lem:weak-q-saturation}
Let $R=\Z_p$, $p\ge 7$, and let $D$ be a horizontal effective Cartier divisor on $\mathcal X_0(3)_R$ supported on $P_\infty,P_0,P_-$. Put
\[
L_D:=H^0(\mathcal X_0(3)_R,\omega^5(\chi_3)(D)).
\]
If $s\in L_D\otimes_R\Q_p$ has Laurent $q$-expansion at $P_\infty$ in $R((q))$, then $s\in L_D$.
\end{lemma}

\begin{proof}
It is enough to prove that the $q$-expansion map
\[
L_D/pL_D\longrightarrow \Fp((q))
\]
is injective. Let $\mathcal L=\omega^5(\chi_3)(D)$. The special fibre of the rigidified stack is integral, with trivial generic stabilizer, and the cusp $P_\infty$ is a non-stacky smooth point. From the reduction exact sequence
\[
0\longrightarrow \mathcal L \xrightarrow{p} \mathcal L \longrightarrow \mathcal L_{\Fp}\longrightarrow 0
\]
the map $L_D/pL_D\hookrightarrow H^0(\mathcal X_{\Fp},\mathcal L_{\Fp})$ is injective. If $\bar s$ has zero Laurent $q$-expansion at $P_\infty$, then its germ in the completed local ring at $P_\infty$ vanishes. Since $\mathcal X_{\Fp}$ is integral and $\mathcal L_{\Fp}$ is invertible, $\bar s$ is zero as a rational section, hence $\bar s=0$. The stacky point $P_-$ only imposes the $\mu_3$-eigenspace condition in the local chart $[\operatorname{Spf}\Fp[[r]]/\mu_3]$ (the line $r^2\Fp[[r^3]]$ for weight $5$ and character $\chi_3$); this is a local restriction and does not produce sections supported at $P_-$ alone.

Now let $s\in L_D\otimes_R\Q_p$ have $q$-expansion in $R((q))$. Choose $a\ge 0$ minimal with $p^a s\in L_D$. If $a>0$, then the $q$-expansion of $p^a s$ is zero modulo $p$. By injectivity, $p^a s\equiv0\pmod p$ in $L_D/pL_D$, contradicting the minimality of $a$. Hence $a=0$, and $s\in L_D$.
\end{proof}

\begin{lemma}[Divisibility of the Hecke finite-difference numerator]\label{lem:divisibility-Nell}
Let $p\ge 7$ be split, let $f\in\{C_0,uC_0\}$, and let $\ell=1,2,3$. Define
\[
N_\ell(f):=\sum_{j=0}^{\ell}(-1)^{\ell-j}\binom{\ell}{j}t^jT_p\!\left(\frac{f}{t^{jp}}\right).
\]
Then there is a horizontal effective Cartier divisor $D_{\ell,p}$ supported on $P_\infty,P_0,P_-$ such that
\[
N_\ell(f)\in p^\ell H^0(\mathcal X_0(3)_{\Z_p},\omega^5(\chi_3)(D_{\ell,p})).
\]
Consequently $\mathfrak E_\ell(f):=N_\ell(f)/p^\ell$ is an integral weakly holomorphic weight-$5$, $\chi_3$ section with polar support contained in $\{P_\infty,P_0,P_-\}$.
\end{lemma}

\begin{proof}
By \Cref{lem:level3-divisor-chart-C0}, the rational function $t=u/(1+27u)^2$ has divisor $P_\infty+P_0-2P_-$. Thus each $f/t^{jp}$ has only finite principal parts at the two cusps and no pole away from them. By \Cref{lem:hecke-weakly-holomorphic}, $T_p(f/t^{jp})$ is an integral weakly holomorphic section. Multiplication by $t^j$ can introduce only poles supported on $P_\infty,P_0,P_-$. Hence, for a sufficiently large horizontal effective Cartier divisor $D_{\ell,p}$ supported on these three points, $N_\ell(f)$ lies in the integral lattice
\[
L_{D_{\ell,p}}:=H^0(\mathcal X_0(3)_{\Z_p},\omega^5(\chi_3)(D_{\ell,p})).
\]

It remains to prove divisibility by $p^\ell$ in this lattice. On the $q$-expansion at $P_\infty$, use $T_p=\Lambda_p+p^4\operatorname{Ver}_p$ from \Cref{lem:hecke-weakly-holomorphic}. For the Cartier part, \Cref{lem:sigma-p-linearity} gives
\[
t^j\Lambda_p\!\left(\frac{f}{t^{jp}}\right)
=\Lambda_p\!\left(\sigma_p(t^j)\frac{f}{t^{jp}}\right)
=\Lambda_p(f\nu_p^j),
\]
where $\nu_p=t(q^p)/t(q)^p=1+\phi_p$. Therefore
\[
\sum_{j=0}^{\ell}(-1)^{\ell-j}\binom{\ell}{j}
t^j\Lambda_p\!\left(\frac{f}{t^{jp}}\right)
=\Lambda_p(f(\nu_p-1)^\ell)=\Lambda_p(f\phi_p^\ell).
\]
Since $\phi_p\in pq\Z_p[[q]]$, this Laurent expansion lies in $p^\ell\Z_p((q))$.

The Verschiebung part is
\[
p^4\sum_{j=0}^{\ell}(-1)^{\ell-j}\binom{\ell}{j}t(q)^j\operatorname{Ver}_p\!\left(\frac{f}{t^{jp}}\right),
\]
which is $p^4$ times an integral Laurent series, hence lies in $p^\ell\Z_p((q))$ for $\ell=1,2,3$. Thus the $q$-expansion of $N_\ell(f)/p^\ell$ lies in $\Z_p((q))$. Applying \Cref{lem:weak-q-saturation} to $N_\ell(f)/p^\ell\in L_{D_{\ell,p}}\otimes\Q_p$ proves the asserted global divisibility.
\end{proof}

\begin{lemma}[Local Cartier--Hecke comparison on the ordinary locus]\label{lem:cartier-hecke-local}
Let $p\ge 7$ be split, let $f\in\{C_0,uC_0\}$, and let $\ell=1,2,3$. Put
\[
\mathfrak E_\ell(f):=\frac{N_\ell(f)}{p^\ell},
\]
where the quotient is the integral global weak section supplied by \Cref{lem:divisibility-Nell}. For $f=C_0$ this is the finite difference $\mathfrak E_\ell$ of \eqref{eq:hecke-fd}; for $f=uC_0$ write $\mathfrak E'_\ell$.
Let $\nu_\Phi:=\Phi^*t/t^p$ and $\psi:=(\nu_\Phi-1)/p$. As meromorphic weight-$5$, $\chi_3$ sections on the ordinary locus, the identity
\[
\mathfrak E_\ell(f)-\mathcal C_{\Phi,5}(f\psi^\ell)=p^{4-\ell}\,\mathcal V_\ell(f)
\]
holds, where
\[
\mathcal V_\ell(f):=\sum_{j=0}^\ell(-1)^{\ell-j}\binom{\ell}{j}t^j\operatorname{Ver}_p(f/t^{jp})
\]
is a meromorphic weight-$5$, $\chi_3$ section on the ordinary locus with $\Z_{(p)}$-integral coefficients. In particular, on the ordinary locus,
\[
\mathfrak E_\ell(f)\equiv\mathcal C_{\Phi,5}(f\psi^\ell)\pmod{p^{4-\ell}}.
\]
On $q$-expansions at the cusp $\infty$ this reads
\[
\mathfrak E_\ell(f)\equiv\Lambda_p\!\left(f\,\frac{\phi_p^\ell}{p^\ell}\right)\pmod{p^{4-\ell}},
\]
because $\nu_\Phi(q)=t(q^p)/t(q)^p$ and $\phi_p=\nu_\Phi-1$ on the Tate chart.
\end{lemma}

\begin{proof}
Target-linearity of the trace: if $a$ is a meromorphic function on the target and $h$ a meromorphic weight-$5$ section on the source, then $\mathcal C_{\Phi,5}((\Phi^*a)h)=a\mathcal C_{\Phi,5}(h)$, because $\Tr_{A/\Phi^*A}((\Phi^*a)b)=a\Tr_{A/\Phi^*A}(b)$. By Katz's construction of the ordinary $U$-operator as normalized trace \cite[\S 3.11]{Katz-padic}, pulled back to the fine atlas $Y$ and descended by \Cref{lem:katz-trace-descent}, the classical prime-to-$3$ Hecke operator decomposes on the ordinary locus as
\[
T_p=\mathcal C_{\Phi,5}+\chi_3(p)p^4\operatorname{Ver}_p,
\]
where $\operatorname{Ver}_p$ is the Verschiebung operator with $q$-expansion $h(q)\mapsto h(q^p)$. At split primes $\chi_3(p)=1$. Expanding $\psi^\ell=p^{-\ell}\sum_{j=0}^\ell(-1)^{\ell-j}\binom{\ell}{j}(\Phi^*t)^j/t^{jp}$, applying $\mathcal C_{\Phi,5}$ and using target-linearity gives
\[
\mathcal C_{\Phi,5}(f\psi^\ell)=\frac{1}{p^\ell}\sum_{j=0}^\ell(-1)^{\ell-j}\binom{\ell}{j}t^j\,\mathcal C_{\Phi,5}\!\left(\frac{f}{t^{jp}}\right).
\]
Replacing $\mathcal C_{\Phi,5}$ by $T_p-p^4\operatorname{Ver}_p$ gives
\[
\mathcal C_{\Phi,5}(f\psi^\ell)=\mathfrak E_\ell(f)-p^{4-\ell}\mathcal V_\ell(f),
\]
which is the displayed local identity. $\mathcal V_\ell(f)$ has integral coefficients since $\operatorname{Ver}_p$ acts by $q^p$-substitution preserving $\Z_{(p)}$-coefficients, and each $t^j$ is integral on the open subscheme where $t$ is regular (and at worst has poles at $P_0$, $P_-$, controlled later). By \Cref{lem:trace-operator-w5}, the $q$-expansion of $\mathcal C_{\Phi,5}$ at $\infty$ is $\Lambda_p$, and on that chart $\nu_\Phi(q)=t(q^p)/t(q)^p$, so $\nu_\Phi-1=\phi_p$. This gives the displayed $q$-expansion form.
\end{proof}

\begin{lemma}[Holomorphy transfer for the finite differences]\label{lem:holomorphy-transfer}
Let $f\in\{C_0,uC_0\}$ and $\ell\in\{1,2,3\}$, and set $S:=\Z_p/p^{4-\ell}\Z_p$. Then $g\,\mathfrak E_\ell(f)$ reduces to an element of $M_5(\Gamma_0(3),\chi_3)\otimes S$, where $g=1+27u$.
\end{lemma}

\begin{proof}
\emph{Step 1: Global divided section.} Define the numerator
\[
N_\ell(f):=\sum_{j=0}^{\ell}(-1)^{\ell-j}\binom{\ell}{j}t^jT_p\!\left(\frac{f}{t^{jp}}\right).
\]
By \Cref{lem:divisibility-Nell}, there is an effective divisor $D_{\ell,p}$ supported on $P_\infty,P_0,P_-$ such that
\[
N_\ell(f)\in p^\ell H^0(\mathcal X_0(3)_{\Z_p},\omega^5(\chi_3)(D_{\ell,p})).
\]
We therefore define
\[
\mathfrak E_\ell(f):=N_\ell(f)/p^\ell
\]
as an integral global weakly holomorphic section. Its polar support is contained in $\{P_\infty,P_0,P_-\}$.

\emph{Step 2: Local holomorphy of $g\,\mathfrak E_\ell(f)$ at $P_\infty$.} On the Tate chart at $\infty$, $\mathfrak E_\ell(f)\equiv\Lambda_p(f\,\phi_p^\ell/p^\ell)\pmod{p^{4-\ell}}$ by \Cref{lem:cartier-hecke-local}. By \Cref{lem:exp-layers,def:nu-phi}, $\phi_p^\ell/p^\ell\in q^\ell\Z_{(p)}[[q]]$, hence $f\phi_p^\ell/p^\ell\in q\,\Z_p[[q]]$, hence $\Lambda_p$ of this lies in $q\,\Z_p[[q]]$, hence $g\,\mathfrak E_\ell(f)\in q\,\Z_p[[q]]\otimes S$ at the cusp $\infty$.

\emph{Step 3: Local holomorphy of $g\,\mathfrak E_\ell(f)$ at $P_0$.} On the Fricke--Tate chart at $0$, by \Cref{lem:cusp-0-regularity-basis}, $\mathcal C_{\Phi,5}(f\psi^\ell)\in q_0\,\Z_p[[q_0]]$. By \Cref{lem:cartier-hecke-local}, the local identity $\mathfrak E_\ell(f)\equiv\mathcal C_{\Phi,5}(f\psi^\ell)\pmod{p^{4-\ell}}$ holds on the cusp-$0$ Tate chart (which is part of the ordinary locus since $p\nmid 3$). Multiplying by $g\in\Z_p[[q_0]]\cdot 27u^{-1}+\Z_p[[q_0]]$ (a function with at worst a simple pole at $P_0$) yields a series in $\Z_p[[q_0]]\otimes S$. The explicit computation in \Cref{lem:cusp-0-regularity-basis} confirms this.

\emph{Step 4: Local holomorphy of $g\,\mathfrak E_\ell(f)$ at $P_-$.} By \Cref{lem:pole-bound-El}, $\operatorname{pole}^{\rm st}_{P_-}(\mathfrak E_\ell(f))\le 1\pmod{p^{4-\ell}}$ for both $f=C_0$ and $f=uC_0$. Since $g$ has a simple zero at $P_-$ (i.e., $\operatorname{ord}^{\rm st}_{P_-}(g)=3$ on the stack and $\ord_{P_-}(g)=1$ on the coarse curve), $g\,\mathfrak E_\ell(f)$ has $\operatorname{pole}^{\rm st}_{P_-}\le 0$ modulo $p^{4-\ell}$, i.e., is regular at $P_-$.

\emph{Conclusion.} Steps 2--4 establish that the polar divisor of $g\,\mathfrak E_\ell(f)$ modulo $p^{4-\ell}$ is empty at each of $P_\infty$, $P_0$, $P_-$. By Step 1 these are the only possible poles. Therefore $g\,\mathfrak E_\ell(f)$ is globally holomorphic of weight $5$ and character $\chi_3$ modulo $p^{4-\ell}$, i.e., it represents an element of $M_5(\Gamma_0(3),\chi_3)\otimes S$.
\end{proof}

\begin{lemma}[Cusp-$0$ regularity of $\phi_p^\ell$]\label{lem:cusp-0-regularity}
Let $q_0$ be the Fricke--Tate parameter at the cusp $0$, normalized so that $w_3^*q=q_0$. Then $t=q_0\eta_0(q_0)$ with $\eta_0(q_0)\in 1+q_0\Z_p[[q_0]]$, and Katz's canonical Frobenius satisfies $\Phi^*q_0=q_0^p$ on the completed local ring at cusp $0$. Consequently
\[
\phi_p=\frac{\Phi^*t}{t^p}-1\in pq_0\Z_p[[q_0]],
\]
and therefore $\phi_p^\ell/p^\ell\in q_0^\ell\Z_p[[q_0]]$ for $\ell\ge 1$.
\end{lemma}

\begin{proof}
At $\infty$ one has $u=q+O(q^2)$ and $t=u/(1+27u)^2=q+O(q^2)$. The Fricke involution exchanges $\infty$ and $0$, and $t$ is $w_3$-invariant. Hence, in the Fricke--Tate parameter $q_0=w_3^*q$ at cusp $0$, $t=q_0\eta_0(q_0)$ with $\eta_0(q_0)\in 1+q_0\Z_p[[q_0]]$. The cusp $0$ is described by the Tate curve $\operatorname{Tate}(q_0)$ in the Fricke chart with its transported $\Gamma_0(3)$-subgroup. Since $p\nmid 3$, the canonical subgroup of $\operatorname{Tate}(q_0)$ is $\mu_p$, and $\operatorname{Tate}(q_0)/\mu_p\simeq\operatorname{Tate}(q_0^p)$. The $\Gamma_0(3)$-subgroup transports isomorphically. Thus Katz's canonical Frobenius is $q_0\mapsto q_0^p$ on the completed local ring at cusp $0$. Now compute
\[
\frac{\Phi^*t}{t^p}=\frac{q_0^p\,\eta_0(q_0^p)}{q_0^p\,\eta_0(q_0)^p}=\frac{\eta_0(q_0^p)}{\eta_0(q_0)^p}.
\]
Since $\eta_0(q_0)\in 1+q_0\Z_p[[q_0]]$, coefficientwise Frobenius gives $\eta_0(q_0)^p\equiv\eta_0(q_0^p)\pmod p$. Therefore $\Phi^*t/t^p\equiv 1\pmod p$. Its constant term is also $1$, so $\phi_p\in p\Z_p[[q_0]]\cap q_0\Z_p[[q_0]]=pq_0\Z_p[[q_0]]$. Raising to the $\ell$-th power and dividing by $p^\ell$ gives $\phi_p^\ell/p^\ell\in q_0^\ell\Z_p[[q_0]]$.
\end{proof}

\begin{lemma}[Cartier expansion at the cusp $0$]\label{lem:cartier-cusp0-expansion}
On the Fricke--Tate chart at the cusp $0$, with parameter $q_0$, the normalized weight-$5$ Katz trace attached to $\Phi$ is the Cartier operator
\[
\mathcal C_{\Phi,5}\!\left(\sum_{n\gg-\infty}a_nq_0^n\right)=\sum_{n\gg-\infty}a_{pn}q_0^n.
\]
In particular it preserves the ideal $q_0\Z_p[[q_0]]$. Moreover, on this chart
\[
\psi=\frac{\Phi^*t/t^p-1}{p}=\frac{1}{p}\!\left(\frac{t(q_0^p)}{t(q_0)^p}-1\right).
\]
\end{lemma}

\begin{proof}
By \Cref{lem:cusp-0-regularity}, Katz's canonical Frobenius is $q_0\mapsto q_0^p$ on the completed local ring at the cusp $0$. The same Tate-curve calculation as in \Cref{lem:trace-operator-w5} applies in the Fricke--Tate coordinate: for the invariant differential $\omega=dT/T$, the quotient by the canonical subgroup satisfies $\Phi^*\omega'=p\omega$, so the weight-$5$ normalization contributes the same factor as at $\infty$. Thus the normalized trace is
\[
\frac{1}{p}\Tr_{\Z_p((q_0))/\Z_p((q_0^p))}.
\]
For $Q_0=q_0^p$, the extension $\Z_p((Q_0))\subset\Z_p((q_0))$ has basis $1,q_0,\ldots,q_0^{p-1}$, and hence
\[
\Tr(q_0^n)=\begin{cases}p\,q_0^{n/p},&p\mid n,\\ 0,&p\nmid n.\end{cases}
\]
Dividing by $p$ gives the displayed Cartier formula, and preservation of $q_0\Z_p[[q_0]]$ is immediate since $pn\ge p\ge 7>0$ forces $n\ge 1$ in any nonzero output term. The formula for $\psi$ is the definition $\psi=(\Phi^*t/t^p-1)/p$ together with $\Phi^*q_0=q_0^p$.
\end{proof}

\begin{lemma}[Cusp-$0$ regularity on the two Eisenstein lines]\label{lem:cusp-0-regularity-basis}
Let $f\in\{C_0,uC_0\}$ and $\ell\in\{1,2,3\}$. Set $\phi_{p,0}:=t(q_0^p)/t(q_0)^p-1$. On the Fricke--Tate chart at the cusp $0$,
\[
f\cdot\frac{\phi_{p,0}^\ell}{p^\ell}\in q_0\,\Z_p[[q_0]].
\]
Consequently the normalized Katz trace satisfies $\mathcal C_{\Phi,5}(f\psi^\ell)\in q_0\,\Z_p[[q_0]]$ on the cusp-$0$ Tate chart, and therefore $g\,\mathcal C_{\Phi,5}(f\psi^\ell)$ is holomorphic at the cusp $0$, where $g=1+27u$ as in the Notation.
\end{lemma}

\begin{proof}
By \Cref{lem:level3-divisor-chart-C0}, $\ord_0(C_0)=1$ and $\ord_0(uC_0)=0$, with $S$-unit leading coefficients. By the proof of \Cref{lem:cusp-0-regularity}, applied in the Fricke--Tate coordinate, the local defect
\[
\phi_{p,0}:=t(q_0^p)/t(q_0)^p-1
\]
lies in $pq_0\Z_p[[q_0]]$, so $\phi_{p,0}^\ell/p^\ell\in q_0^\ell\Z_p[[q_0]]$ for $\ell\ge 1$. Multiplying:
\[
\ord_0\!\left(C_0\cdot\phi_{p,0}^\ell/p^\ell\right)\ge 1+\ell\ge 2,\qquad
\ord_0\!\left(uC_0\cdot\phi_{p,0}^\ell/p^\ell\right)\ge \ell\ge 1.
\]
In particular both products lie in $q_0\Z_p[[q_0]]$. By \Cref{lem:cartier-cusp0-expansion}, the normalized Katz trace on the cusp-$0$ chart is the Cartier operator for $q_0\mapsto q_0^p$, and it preserves $q_0\Z_p[[q_0]]$; since on this chart $\psi=\phi_{p,0}/p$ by the same lemma, this gives $\mathcal C_{\Phi,5}(f\psi^\ell)\in q_0\Z_p[[q_0]]$.

Finally, $\ord_0(g)=\ord_0(1+27u)=-1$ because $\ord_0(u)=-1$ with $S$-unit leading coefficient (\Cref{lem:level3-divisor-chart-C0}). Hence multiplying an element of $q_0\Z_p[[q_0]]$ by $g$ yields an element of $\Z_p[[q_0]]$, i.e., a holomorphic $q_0$-series at the cusp $0$.
\end{proof}

\subsection{Classicality on both basis elements}\label{subsec:classicality-both}

\begin{lemma}[Layer decomposition of $\Phi^*t/t^p-1$ at $P_-$]\label{lem:psi-layer-decomposition}
On the formal completion at $P_-$ with stack parameter $r$ and $s=r^3$, write
\[
\Phi^*r=F(r)=r^p+\sum_{j\ge1}p^j r\alpha_j(r^3),
\qquad \alpha_j\in\Z_p[\zeta_3][[s]].
\]
Then
\[
\frac{\Phi^*t}{t^p}-1=pA_1(r)+p^2A_2(r)+p^3A_3(r)+O(p^4),
\]
where $A_j(r)\in r^{j(1-p)}\Z_p[\zeta_3][[r^3]]$ for $j=1,2,3$.
\end{lemma}

\begin{proof}
With the normalization $g=1+27u=r^3=s$, one has
\[
t=\frac{u}{g^2}=r^{-6}\tau(s),\qquad \tau(s)=\frac{s-1}{27}\in\Z_p[\zeta_3][[s]]^\times.
\]
Put
\[
Y:=\frac{F(r)}{r^p}-1=pY_1+p^2Y_2+p^3Y_3+O(p^4),
\qquad Y_j\in r^{1-p}\Z_p[\zeta_3][[s]].
\]
Then
\[
\frac{\Phi^*t}{t^p}=(1+Y)^{-6}\cdot \frac{\tau(F(r)^3)}{\tau(s)^p}.
\]
The first factor has the binomial expansion
\[
(1+Y)^{-6}=1+pB_1+p^2B_2+p^3B_3+O(p^4),
\qquad B_j\in r^{j(1-p)}\Z_p[\zeta_3][[s]].
\]
For the second factor, since $F(r)^3=s^p(1+Y)^3$ and $\tau(s)=(s-1)/27$,
\[
\frac{\tau(F(r)^3)}{\tau(s)^p}=27^{p-1}\frac{s^p(1+Y)^3-1}{(s-1)^p}.
\]
The term with $Y=0$ is congruent to $1$ modulo $p$, because $s^p-1\equiv (s-1)^p\pmod p$ and $27^{p-1}\equiv1\pmod p$; its higher $p$-adic coefficients lie in $\Z_p[\zeta_3][[s]]$, which is contained in $r^{j(1-p)}\Z_p[\zeta_3][[s]]$ for every $j\ge1$. The remaining terms contain $(1+Y)^3-1$, whose degree-$j$ contribution lies in $p^j r^{j(1-p)}\Z_p[\zeta_3][[s]]$; multiplication by $s^p/(s-1)^p$ preserves these inclusions. Hence
\[
\frac{\tau(F(r)^3)}{\tau(s)^p}=1+pC_1+p^2C_2+p^3C_3+O(p^4),
\qquad C_j\in r^{j(1-p)}\Z_p[\zeta_3][[s]].
\]
Multiplying the two expansions gives the asserted layer decomposition.
\end{proof}

\begin{lemma}[Stack pole bound on the two Eisenstein lines]\label{lem:pole-bound-El}
Let $f\in\{C_0,uC_0\}$, and let $\mathfrak E_\ell(f)$ be the finite difference of \Cref{lem:cartier-hecke-local}. Then, for $\ell=1,2,3$,
\[
\operatorname{pole}^{\rm st}_{P_-}\bigl(\mathfrak E_\ell(f)\bigr)\le 1\pmod{p^{4-\ell}}.
\]
\end{lemma}

\begin{proof}
Set $\psi:=(\Phi^*t/t^p-1)/p$. By \Cref{lem:psi-layer-decomposition},
\[
\psi\in r^{1-p}\Z_p[\zeta_3][[r^3]]+pr^{2(1-p)}\Z_p[\zeta_3][[r^3]]+p^2r^{3(1-p)}\Z_p[\zeta_3][[r^3]]\pmod{p^3}.
\]
At $P_-$, the divisor chart gives
\[
C_0=r^2c(r^3)e,\qquad uC_0=r^2c_u(r^3)e,
\]
with $c(0),c_u(0)\in\Z_p^\times$, since $u$ is a unit at $P_-$ (\Cref{lem:level3-divisor-chart-C0}). Thus the same estimate applies to both $f=C_0$ and $f=uC_0$. Hence $f\psi^\ell=\sum_{s=\ell}^3 p^{s-\ell}F_{\ell,s}(r)e\pmod{p^{4-\ell}}$, $F_{\ell,s}\in r^{2+s(1-p)}\Z_p[[r^3]]$. By \Cref{prop:pole-estimate-star}, each summand maps under $\mathcal C_{\Phi,5}$ into $r^{-1}\Z_p[[r^3]]e\pmod{p^{4-\ell}}$. By the local Cartier--Hecke comparison \Cref{lem:cartier-hecke-local}, on the ordinary neighbourhood of $P_-$ (which is ordinary at split primes since $p\equiv 1\pmod 3$ makes $j=0$ split-ordinary),
\[
\mathfrak E_\ell(f)\equiv \mathcal C_{\Phi,5}(f\psi^\ell)\pmod{p^{4-\ell}}.
\]
Hence the $r^{-1}$-pole bound established for $\mathcal C_{\Phi,5}(f\psi^\ell)$ above transfers to $\mathfrak E_\ell(f)$ modulo $p^{4-\ell}$. The displayed $q$-expansion identity is the Tate-chart specialization at the cusp $\infty$ and is not used here.
\end{proof}

\begin{lemma}[Level-$3$ divisor chart for $C_0$ and $uC_0$]\label{lem:level3-divisor-chart-C0}
Let $p\ge 7$ be split, $p\equiv 1\pmod 3$, let $R=\Z_p$, and let $S$ be either $R$ or $R/p^N R$. Let $P_\infty$ denote the cusp with $q=0$ so that $u(P_\infty)=0$, let $P_0$ denote the other cusp so that $u=\infty$, and let $P_-:1+27u=0$. On the coarse curve $X_0(3)_S\simeq\mathbb P^1_{u,S}$ the following local order table holds:
\[
\begin{array}{c|ccc}
& P_\infty & P_0 & P_- \\ \hline
u & 1 & -1 & 0\\
C_0 & 0 & 1 & 2\ \text{\rm(stack)}\\
uC_0 & 1 & 0 & 2\ \text{\rm(stack)}.
\end{array}
\]
Moreover the leading coefficients in all three local descriptions are $S$-units. In particular $C_0$ has no zero away from $P_0$ and $P_-$, and $uC_0$ is holomorphic at all three special points.
\end{lemma}

\begin{proof}
This is \Cref{lem:prelim-divisor-chart}.
\end{proof}

\begin{lemma}[Integral basis and mod $p^N$ classicality on $M_5(\Gamma_0(3),\chi_3)$]\label{lem:M5-integral-mod-pN}
Let $p\ge 7$ be split, $p\equiv 1\pmod 3$, and let $R=\Z_p$. For $S=R$ or $S=R/p^N R$, put
\[
\mathcal X_S:=\mathcal X_0(3)\times_{\operatorname{Spec} R}\operatorname{Spec} S,\qquad \mathcal L_S:=\omega^5(\chi_3)|_{\mathcal X_S},\qquad M_5(S):=H^0(\mathcal X_S,\mathcal L_S).
\]
Then the $S$-linear map $S^2\to M_5(S)$, $(a,b)\mapsto aC_0+b\,uC_0$, is an isomorphism. Equivalently, $M_5(S)=S\,C_0\oplus S\,uC_0$. In particular for every $N\ge 1$, $M_5(R/p^N R)=(R/p^N R)\,C_0\oplus(R/p^N R)\,uC_0$, and the $q$-expansion map at $P_\infty$ is injective on $M_5(R/p^N R)$, with image the $(R/p^N R)$-span of $C_0(q)$ and $u(q)C_0(q)$.
\end{lemma}

\begin{proof}
By \Cref{lem:level3-divisor-chart-C0} both $C_0$ and $uC_0$ are holomorphic sections of $\mathcal L_S$. Let $f\in M_5(S)$. We show $f/C_0$ is a section of $\mathcal O_{\mathbb P^1_S}(P_0)$ on the coarse curve. On the non-stacky points away from the zeros of $C_0$, $f/C_0$ is regular. The only finite zero of $C_0$ is at $P_-$. Near $P_-$ we use the quotient-stack chart $\widehat{\mathcal X}_{S,P_-}\simeq[\operatorname{Spf}S[[r]]/\mu_3]$, $s=r^3$. Choose the local generator $e$ of $\mathcal L_S$ so that $C_0=r^2c(s)e$, $c(0)\in S^\times$. If $f=F(r)e$ is an invariant local section, then $F(r)$ lies in the same $\mu_3$-eigenspace as $r^2c(s)$, because $f$ and $C_0$ are sections of the same equivariant line bundle. Since $S[[r]]=S[[s]]\oplus rS[[s]]\oplus r^2 S[[s]]$ and $3\in S^\times$, this eigenspace is exactly $r^2 S[[s]]$. Thus $F(r)=r^2 d(s)$, $d(s)\in S[[s]]$, and $f/C_0=d(s)/c(s)\in S[[s]]$. Therefore $f/C_0$ is regular at $P_-$.

At $P_\infty$ the section $C_0$ is a unit, so $f/C_0$ is regular. On the affine coarse chart $\mathbb A^1_{u,S}=\operatorname{Spec} S[u]$, $f/C_0\in S[u]$. At $P_0$, write $v=1/u$. By \Cref{lem:level3-divisor-chart-C0}, $C_0=v\,c_0(v)e_0$, $c_0(0)\in S^\times$. For $f=F_0(v)e_0$ with $F_0(v)\in S[[v]]$, $v\cdot f/C_0=F_0(v)/c_0(v)\in S[[v]]$. Thus $f/C_0$ has pole order at most $1$ at $P_0$ and no other pole.

It remains to compute $H^0(\mathbb P^1_S,\mathcal O(P_0))=S[u]\cap v^{-1}S[v]$. Since $v^{-1}S[v]=S[u^{-1}]+S\cdot u$ and the Laurent monomials are $S$-linearly independent, the intersection is $S\oplus S\cdot u$. Hence $f/C_0=a+bu$, $a,b\in S$, and $f=aC_0+b\,uC_0$. This proves surjectivity.

For injectivity and the $q$-expansion statement, use $C_0(q)=1+O(q)$ and $u(q)C_0(q)=q+O(q^2)$. The matrix $\begin{pmatrix}[q^0]C_0 & [q^0](uC_0)\\ [q^1]C_0 & [q^1](uC_0)\end{pmatrix}$ has $[q^0](uC_0)=0$ and $[q^0]C_0=[q^1](uC_0)=1$, hence determinant $1$. Thus if $aC_0+buC_0$ has zero $q$-expansion over $S$, the constant term gives $a=0$, and the $q^1$ coefficient then gives $b=0$. Taking $S=R$ gives the integral free $R$-basis, and taking $S=R/p^N R$ gives the asserted mod $p^N$ classicality.
\end{proof}

\begin{theorem}[Closure for $C_0$]\label{thm:KB-closed}
For every split prime $p\ge 7$ and every $\ell=1,2,3$,
\[
g\Lambda_p(C_0\phi_p^\ell/p^\ell)\in(\Z_p/p^{4-\ell}\Z_p)\cdot uC_0.
\]
\end{theorem}

\begin{proof}
Set $S=R/p^{4-\ell}R$. By \Cref{lem:cartier-hecke-local} it suffices to work with $\mathfrak E_\ell$. By \Cref{lem:holomorphy-transfer}, $g\mathfrak E_\ell\in M_5(\Gamma_0(3),\chi_3)\otimes S$. Its constant term at $P_\infty$ is zero, because the cusp-$\infty$ analysis in the proof of \Cref{lem:holomorphy-transfer} shows $g\mathfrak E_\ell\in q\Z_p[[q]]\otimes S$. By \Cref{lem:M5-integral-mod-pN}, $g\mathfrak E_\ell=aC_0+b\,uC_0$ with $a,b\in S$. The constant term at $P_\infty$ gives $a=0$, since $[q^0]C_0=1$ and $[q^0](uC_0)=0$. Hence $g\mathfrak E_\ell\in S\cdot uC_0$, and \Cref{lem:cartier-hecke-local} gives the stated congruence for $g\Lambda_p(C_0\phi_p^\ell/p^\ell)$.
\end{proof}

\begin{lemma}[Parallel closure for $uC_0$]\label{lem:parallel-KB-closure}
For every split prime $p\ge 7$ and every $\ell=1,2,3$,
\[
g\Lambda_p(uC_0\phi_p^\ell/p^\ell)\in(\Z_p/p^{4-\ell}\Z_p)C_0\oplus(\Z_p/p^{4-\ell}\Z_p)uC_0.
\]
\end{lemma}

\begin{proof}
Set $S=R/p^{4-\ell}R$, and let $\mathfrak E'_\ell$ be the analogue of \eqref{eq:hecke-fd} with $uC_0$ replacing $C_0$. By \Cref{lem:holomorphy-transfer} applied with $f=uC_0$, $g\mathfrak E'_\ell\in M_5(\Gamma_0(3),\chi_3)\otimes S$. Therefore \Cref{lem:M5-integral-mod-pN} gives $g\mathfrak E'_\ell\in S\,C_0\oplus S\,uC_0$. The Cartier--Hecke comparison \Cref{lem:cartier-hecke-local} extends with $uC_0$ in place of $C_0$, giving the asserted membership for $g\Lambda_p(uC_0\phi_p^\ell/p^\ell)$.
\end{proof}

\subsection{Bridge scalar via Atkin--Lehner}\label{subsec:bridge}

By \Cref{thm:KB-closed,lem:parallel-KB-closure} together with the triangular conversion \Cref{lem:phi-U-triangular}, there are scalars $\alpha_\ell,\beta_\ell,\gamma_\ell\in\Z_p$ modulo $p^{4-\ell}$ such that
\begin{align}
g\Lambda_p(C_0U_p^\ell)&\equiv p^\ell\gamma_\ell\,uC_0\pmod{p^4},\label{eq:C0-classicality}\\
g\Lambda_p(uC_0U_p^\ell)&\equiv p^\ell\alpha_\ell C_0+p^\ell\beta_\ell\,uC_0\pmod{p^4}.\label{eq:uC0-classicality}
\end{align}
The constant term at $\infty$ of the left side of \eqref{eq:uC0-classicality} is zero, so $\alpha_\ell\equiv 0\pmod{p^{4-\ell}}$.

\begin{lemma}[Weakly holomorphic Atkin--Lehner intertwining]\label{lem:atkin-lehner-intertwining}
Let $M_5^!(\Gamma_0(3),\chi_3)$ denote the space of weight-$5$, $\chi_3$ forms which are meromorphic only at the cusps. The Fricke pullback $w_3^*$ preserves this space, and for every prime $p\nmid 3$,
\[
T_p w_3^*=\chi_3(p)\,w_3^*T_p
\]
on $M_5^!(\Gamma_0(3),\chi_3)$. With the geometric normalization used here, $w_3^*C_0=-27\,uC_0$ and $w_3^*(uC_0)=-C_0/27$, and $t$ is $w_3$-invariant. Hence at split primes $\chi_3(p)=1$,
\begin{equation}\label{eq:fricke-intertwining}
\mathfrak E'_\ell=-\frac{1}{27}\,w_3^*\mathfrak E_\ell,\qquad\ell=1,2,3.
\end{equation}
\end{lemma}

\begin{proof}
The operators $w_3^*$ and $T_p$ are defined by slash operators attached to matrices in $\operatorname{GL}_2^+(\Q)$. These slash operators act on meromorphic modular forms just as on holomorphic modular forms; finite cusp principal parts remain finite cusp principal parts because the relevant maps are finite on the compact modular curve. Thus both operators preserve $M_5^!(\Gamma_0(3),\chi_3)$. The identity $T_p w_3^*=\chi_3(p)w_3^*T_p$ is the original Atkin--Lehner double-coset identity \cite[\S 5, equation~(58) and Theorem~3]{AtkinLehner70}, established by computing the product of double cosets $\Gamma_0(3)\bigl(\begin{smallmatrix}0&-1\\3&0\end{smallmatrix}\bigr)\Gamma_0(3)$ and $\Gamma_0(3)\bigl(\begin{smallmatrix}1&0\\0&p\end{smallmatrix}\bigr)\Gamma_0(3)$ in the Hecke algebra. The proof is an equality of finite sums of slash operators attached to double cosets, derived in a Hecke algebra independent of vanishing or holomorphy at the cusps; only the transformation law and the finite-dimensionality of cusp principal parts are used. The same calculation thus holds on the weakly holomorphic space $M_5^!(\Gamma_0(3),\chi_3)$ as on the holomorphic subspace; for the explicit weakly-holomorphic extension of the Hecke formalism see \cite[\S 2]{Guerzhoy08}.

The Fricke action on $u$ is $w_3^*u=1/(729u)$, and therefore
\[
w_3^*t=\frac{w_3^*u}{(1+27w_3^*u)^2}=\frac{1/(729u)}{(1+1/(27u))^2}=\frac{u}{(1+27u)^2}=t.
\]
The eta transformation $\eta(-1/\tau)=(-i\tau)^{1/2}\eta(\tau)$ with the geometric Fricke normalization gives $w_3^*C_0=-27q+O(q^2)$ at the cusp $\infty$, and $\operatorname{ord}_\infty(w_3^*C_0)=\operatorname{ord}_0(C_0)=1$. Hence $w_3^*C_0$ has zero constant term. By \Cref{lem:M5-integral-mod-pN} over $R=\Z_p$, the holomorphic weight-$5$, $\chi_3$ form $w_3^*C_0$ is an $R$-linear combination of $C_0$ and $uC_0$; the zero constant term kills the $C_0$ component, and the $q^1$ coefficient gives the scalar $-27$. Thus $w_3^*C_0=-27uC_0$. Applying $w_3^2=\mathrm{id}$ gives $w_3^*(uC_0)=-C_0/27$.

We apply the weakly-holomorphic Atkin--Lehner intertwining only to the individual terms $T_p(C_0/t^{jp})$ and $T_p(uC_0/t^{jp})$, which are weakly holomorphic with finite cusp principal parts. The finite-difference combination $\mathfrak E_\ell(f)$ may have a pole at $P_-$ before $g$-multiplication; the intertwining is not applied to $\mathfrak E_\ell(f)$ as an independent weakly holomorphic input.

Finally, apply $w_3^*$ termwise to the finite difference defining $\mathfrak E_\ell=\mathfrak E_\ell(C_0)$. Since $t$ is $w_3$-invariant and $\chi_3(p)=1$ for split primes,
\[
\begin{aligned}
w_3^*\mathfrak E_\ell&=\frac{1}{p^\ell}\sum_{j=0}^\ell(-1)^{\ell-j}\binom{\ell}{j}t^j\,w_3^*T_p\!\left(\frac{C_0}{t^{jp}}\right)\\
&=\frac{1}{p^\ell}\sum_{j=0}^\ell(-1)^{\ell-j}\binom{\ell}{j}t^j\,T_p w_3^*\!\left(\frac{C_0}{t^{jp}}\right)\\
&=-27\,\frac{1}{p^\ell}\sum_{j=0}^\ell(-1)^{\ell-j}\binom{\ell}{j}t^j\,T_p\!\left(\frac{uC_0}{t^{jp}}\right)=-27\,\mathfrak E'_\ell.
\end{aligned}
\]
This is the asserted relation.
\end{proof}

\begin{remark}[Split-prime trivialization of the twist]\label{rem:split-trivialization}
The split-prime hypothesis $p\equiv 1\pmod 3$ enters the Cartier--Hecke chain in two distinct places: (i) as the input condition for the $\mu_3$-equivariant canonical Frobenius normal form in \Cref{prop:local-frobenius-normal-form}, and (ii) as the condition $\chi_3(p)=1$ that makes the Atkin--Lehner intertwining $T_pw_3^*=\chi_3(p)w_3^*T_p$ untwisted. The bridge identity in \Cref{cor:bridge-scalar} relies specifically on the second trivialization. At inert primes $\chi_3(p)=-1$ would sign-flip the intertwining and break the bridge cancellation; this matches the $q$-side obstruction in \Cref{sec:inert}.
\end{remark}

\begin{corollary}[Bridge scalar cancellation]\label{cor:bridge-scalar}
For every split $p\ge 7$ and every $\ell=1,2,3$,
\[
\beta_\ell\equiv 27^{-1}\gamma_\ell\pmod{p^{4-\ell}},\qquad\text{equivalently}\qquad \gamma_\ell-27\beta_\ell\equiv 0\pmod{p^{4-\ell}}.
\]
\end{corollary}

\begin{proof}
By \Cref{thm:KB-closed,lem:parallel-KB-closure} there are scalars $\Gamma_\ell, B_\ell\in\Z_p/p^{4-\ell}\Z_p$ with $g\mathfrak E_\ell\equiv\Gamma_\ell uC_0$ and $g\mathfrak E'_\ell\equiv B_\ell uC_0\pmod{p^{4-\ell}}$ (the $C_0$-coefficient on the second side vanishes from constant-term considerations). The division by $g$ below is performed in the common meromorphic lattice with possible pole at $P_-$, since $g$ has a zero of order $1$ on the coarse curve there. After multiplying back by $g$, the resulting congruence lies in the holomorphic $M_5$-lattice modulo $p^{4-\ell}$. Using $w_3^*g=g/(27u)$, $w_3^*(uC_0)=-C_0/27$:
\[
w_3^*\mathfrak E_\ell\equiv\Gamma_\ell\frac{w_3^*(uC_0)}{w_3^*g}=\Gamma_\ell\frac{-C_0/27}{g/(27u)}=-\Gamma_\ell\frac{uC_0}{g}\pmod{p^{4-\ell}}.
\]
Combining with \eqref{eq:fricke-intertwining} and multiplying by $g$ yields $g\mathfrak E'_\ell\equiv\Gamma_\ell/27\cdot uC_0\pmod{p^{4-\ell}}$, so $B_\ell\equiv 27^{-1}\Gamma_\ell$. Since the scalar $27^{-1}$ is common to all three $\phi_p$-layers, \Cref{lem:phi-U-triangular} transfers the vector proportionality $B_\bullet\equiv 27^{-1}\Gamma_\bullet$ to the $U_p$-layer scalars, giving $\beta_\ell\equiv 27^{-1}\gamma_\ell$.
\end{proof}

\subsection{Mixed Cartier cancellation, transport, and the main theorem}\label{subsec:transport-and-main}

The main result of this paper, \Cref{thm:split-final-closed}, depends only on \Cref{prop:mixed-cartier-cancellation} below, the Katz--Dwork reduction \Cref{thm:KD-input}, and the Lagrange--B\"urmann identity \Cref{thm:LB}. The companion Dwork congruence \cite[Thm.~7.8]{Shvets-companion} for the diagonal point $(\tfrac13,\tfrac13;1)$ is used only in the optional \Cref{thm:transport-cancel} (transport comparison with the base-side layers $C_0V_p^\ell$) and is not invoked anywhere in the proof of the main theorem.

\begin{proposition}[Mixed Cartier cancellation]\label{prop:mixed-cartier-cancellation}
For every split prime $p\ge 7$ and every $\ell=1,2,3$,
\[
\Lambda_p(C_{\mathrm{mix}}U_p^\ell)\equiv 0\pmod{p^4}.
\]
\end{proposition}

\begin{proof}
Apply $\Lambda_p$ to $C_{\mathrm{mix}}U_p^\ell=C_0U_p^\ell-27uC_0\,U_p^\ell$ and multiply by $g$. Using \eqref{eq:C0-classicality}, \eqref{eq:uC0-classicality}, and $\alpha_\ell\equiv 0$,
\[
g\Lambda_p(C_{\mathrm{mix}}U_p^\ell)\equiv p^\ell(\gamma_\ell-27\beta_\ell)uC_0\pmod{p^4}.
\]
By \Cref{cor:bridge-scalar}, $\gamma_\ell-27\beta_\ell\equiv 0\pmod{p^{4-\ell}}$, hence $p^\ell(\gamma_\ell-27\beta_\ell)\equiv 0\pmod{p^4}$, so $g\Lambda_p(C_{\mathrm{mix}}U_p^\ell)\equiv 0\pmod{p^4}$. Since $g=1+O(q)$ is a unit in $\Z_p[[q]]$,
\[
\Lambda_p(C_{\mathrm{mix}}U_p^\ell)\equiv 0\pmod{p^4}.\qedhere
\]
\end{proof}

\begin{theorem}[Split-prime supercongruence]\label{thm:split-final-closed}
For every split prime $p\ge 7$, $p\equiv 1\pmod 3$, and every $m\ge 1$,
\[
A_{mp}^{\mathrm{mix}}\equiv A_m^{\mathrm{mix}}\pmod{p^4}.
\]
\end{theorem}

\begin{proof}
By \Cref{prop:mixed-cartier-cancellation}, $\Lambda_p(C_{\mathrm{mix}}U_p^\ell)\equiv 0\pmod{p^4}$ for $\ell=1,2,3$. \Cref{thm:KD-input} then gives, for every $m\ge 1$, the scalar Katz--Dwork congruence
\[
\Lambda_p(C_{\mathrm{mix}}H_{\mathrm{mix}}^{pm})\equiv C_{\mathrm{mix}}H_{\mathrm{mix}}^m\pmod{p^4}.
\]
Extracting the coefficient of $q^m$ and applying Lagrange--B\"urmann (\Cref{thm:LB}) gives $A_{mp}^{\mathrm{mix}}\equiv A_m^{\mathrm{mix}}\pmod{p^4}$.
\end{proof}

\begin{theorem}[Conditional transport comparison]\label{thm:transport-cancel}
Assume the truncated Dwork congruence at $(\tfrac13,\tfrac13;1)$ from \cite[Thm.~7.8]{Shvets-companion}. For every split prime $p\ge 7$, every $\ell\in\{1,2,3\}$, and every $n\ge 1$,
\[
[q^{np}]C_{\mathrm{mix}}U_p^\ell\equiv[q^{np}]C_0V_p^\ell\pmod{p^4}.
\]
\end{theorem}

\begin{proof}
\Cref{prop:mixed-cartier-cancellation} gives $\Lambda_p(C_{\mathrm{mix}}U_p^\ell)\equiv 0\pmod{p^4}$. The companion theorem gives $\Lambda_p(C_0H_0^{pX})\equiv C_0H_0^X\pmod{(p^4,X^4)}$, equivalent to $\Lambda_p(C_0V_p^\ell)\equiv 0\pmod{p^4}$ for $\ell=1,2,3$. Subtracting,
\[
\Lambda_p(C_{\mathrm{mix}}U_p^\ell-C_0V_p^\ell)\equiv 0\pmod{p^4},
\]
and the $q^n$-coefficient is exactly $[q^{np}]\bigl(C_{\mathrm{mix}}U_p^\ell-C_0V_p^\ell\bigr)$.
\end{proof}

\begin{remark}[Status of external inputs]\label{rem:external-inputs}
We summarize the dependencies of \Cref{sec:WC}. The $\mu_3$-equivariance of the Frobenius lift at $P_-$ follows from functoriality of Katz's canonical subgroup under the $j=0$ stabilizer (\Cref{lem:mu3-preserves-canonical-subgroup,lem:katz-trace-descent}); the descended Katz trace operator is constructed internally on the ordinary locus, and its $q$-expansion at the cusp $\infty$ and at the cusp $0$ is computed explicitly (\Cref{lem:trace-operator-w5,lem:cartier-cusp0-expansion}); the coefficient extraction at $P_-$ uses the trace-residue formula recorded in \Cref{lem:trace-residue-tame}; cusp-$0$ regularity of $\phi_p$ and of the two Eisenstein lines $C_0\psi^\ell$, $uC_0\psi^\ell$ is checked directly on the Fricke--Tate chart (\Cref{lem:cusp-0-regularity,lem:cusp-0-regularity-basis}); the local versus global Cartier--Hecke comparison is separated explicitly (\Cref{lem:cartier-hecke-local,lem:holomorphy-transfer}); the weakly holomorphic Atkin--Lehner application is reduced to a double-coset identity acting on finite cusp principal parts (\Cref{lem:atkin-lehner-intertwining}); the unit leading coefficient at $P_-$ is proved from the theta representation and the local unit computation (\Cref{lem:C0-theta-representation,lem:Pminus-unit-leading}); and mod $p^N$ classicality is established directly through the coarse curve $\mathbb P^1_u$ together with the quotient-stack chart at $P_-$, without invoking any cohomological base-change theorem (\Cref{lem:level3-divisor-chart-C0,lem:M5-integral-mod-pN}). The divided finite differences are made integral by the elementary stacky weak $q$-expansion saturation at the non-stacky cusp $P_\infty$ (\Cref{lem:weak-q-saturation,lem:divisibility-Nell}); this uses only integrality of the special fibre and does not invoke Katz's canonical $q$-expansion principle or any external mod-$p^N$ machinery. The remaining external inputs are: Katz's construction and functoriality of the canonical subgroup and ordinary normalized trace \cite[\S\S 3.10--3.11]{Katz-padic}; the tame quotient-stack local form at $P_-$ \cite[Theorem~3.2]{AOV}; the trace-residue formula for finite flat extensions of complete one-dimensional local rings \cite[\S 3]{Tate-residues}, \cite[Ch.~III]{Hartshorne-RD}; the Atkin--Lehner--Hecke double-coset identity \cite{AtkinLehner70}; and the Hecke theory of weakly holomorphic modular forms with finite cusp principal parts \cite{Guerzhoy08}.
\end{remark}

\section{Inert-prime obstruction}\label{sec:inert}

Define $C_{\mathrm{mix}}(q)=1+\sum_{n\ge 1}c_n^{\mathrm{mix}}q^n$. Set
\[
s(n):=\sum_{d\mid n}\chi_3(d)d^4,\qquad\beta(n):=\sum_{d\mid n}\chi_3(n/d)d^4.
\]
Then $c_n^{\mathrm{mix}}=3s(n)-27\beta(n)$.

\begin{theorem}[Split-prime Eisenstein tower]\label{thm:split-tower}
If $p\ge 5$ and $\chi_3(p)=1$, then for all $m,r\ge 1$, $c_{mp^r}^{\mathrm{mix}}\equiv c_{mp^{r-1}}^{\mathrm{mix}}\pmod{p^{4r}}$.
\end{theorem}

\begin{proof}
Both $s=1*(\chi_3\,\mathrm{id}^4)$ and $\beta=\chi_3*\mathrm{id}^4$ are multiplicative. For $p\ne 3$,
\[
s(p^N)=\sum_{j=0}^N\chi_3(p)^jp^{4j},\qquad\beta(p^N)=\sum_{j=0}^N\chi_3(p)^{N-j}p^{4j}.
\]
At $\chi_3(p)=1$ both equal $1+p^4+\cdots+p^{4N}$. Writing $m=p^am_0$ with $(m_0,p)=1$,
$s(mp^r)-s(mp^{r-1})=p^{4(a+r)}s(m_0)$, and similarly for $\beta$, giving the claim.
\end{proof}

\begin{theorem}[Inert-prime $q$-side obstruction]\label{thm:inert-obstruction}
Let $p\ge 5$, $\chi_3(p)=-1$. Then $\beta(p)-\beta(1)=p^4-2\not\equiv 0\pmod p$, and more generally if $m=p^am_0$ with $p\nmid m_0$ and $\beta(m_0)\not\equiv 0\pmod p$,
\[
\vp(c_{mp}^{\mathrm{mix}}-c_m^{\mathrm{mix}})=0.
\]
The obstruction comes from $\sum_N\beta(p^N)X^N=1/((1-\chi_3(p)X)(1-p^4X))$, whose unit-root factor is $(1+X)^{-1}$ at inert primes.
\end{theorem}

\begin{proof}
For $\chi_3(p)=-1$, $\beta(1)=1$ and $\beta(p)=-1+p^4$, so $\beta(p)-\beta(1)=p^4-2$. For general $m=p^am_0$,
\[
\beta(mp)-\beta(m)=\bigl(\beta(p^{a+1})-\beta(p^a)\bigr)\beta(m_0).
\]
Modulo $p$, $\beta(p^N)\equiv(-1)^N$, so the parenthesis is $-2(-1)^a\pmod p$, a unit. The $s$-part has valuation $\ge 4$, so the mixed coefficient difference has valuation $0$ whenever $\beta(m_0)$ is a unit.
\end{proof}

\begin{proposition}[No split-style formal $q$-congruence at inert primes]\label{prop:no-inert-dwork}
If $p\ge 5$ is inert, then $\Lambda_p(C_{\mathrm{mix}}H_{\mathrm{mix}}^{pX})\equiv C_{\mathrm{mix}}H_{\mathrm{mix}}^X\pmod{(p^4,X^4)}$ cannot hold.
\end{proposition}

\begin{proof}
Set $X=0$: the congruence implies $c_{mp}^{\mathrm{mix}}\equiv c_m^{\mathrm{mix}}\pmod{p^4}$ for every $m\ge 1$, contradicting \Cref{thm:inert-obstruction}.
\end{proof}

For the coefficient-level statement, it is useful to keep both Eisenstein branches visible. For \(\varepsilon\in\{\pm1\}\), define
\[
C^{(\varepsilon)}(q):=C_0(q)-27\varepsilon\,u(q)C_0(q),
\qquad
A_m^{(\varepsilon)}:=[q^m]C^{(\varepsilon)}(q)H_{\mathrm{mix}}(q)^m.
\]
Thus \(A_m^{(+1)}=A_m^{\mathrm{mix}}\), while \(A_m^{(-1)}\) is the conjugate Eisenstein branch obtained by changing the sign of the second Eisenstein component.

\begin{theorem}[Inert Cartier parity on the coefficient sequence]\label{thm:inert-cartier-parity}
For every prime \(p\ge5\), every \(r\ge0\), and every \(m\ge0\),
\[
A_{mp^r}^{\mathrm{mix}}\equiv A_m^{(\chi_3(p)^r)}\pmod p.
\]
Consequently, if \(p\equiv1\pmod3\), then
\[
A_{mp^r}^{\mathrm{mix}}\equiv A_m^{\mathrm{mix}}\pmod p,
\]
whereas if \(p\equiv2\pmod3\), then
\[
A_{mp^{2r}}^{\mathrm{mix}}\equiv A_m^{\mathrm{mix}}\pmod p,
\qquad
A_{mp^{2r+1}}^{\mathrm{mix}}\equiv A_m^{(-1)}\pmod p.
\]
\end{theorem}

\begin{proof}
Write \(\Lambda_p^r\) for the \(r\)-fold Cartier operator. By \Cref{thm:LB},
\[
A_{mp^r}^{\mathrm{mix}}=[q^{mp^r}]C_{\mathrm{mix}}(q)H_{\mathrm{mix}}(q)^{mp^r}.
\]
Since \(H_{\mathrm{mix}}\in1+q\Z[[q]]\), the freshman Frobenius congruence gives
\[
H_{\mathrm{mix}}(q)^{mp^r}\equiv H_{\mathrm{mix}}(q^{p^r})^m\pmod p.
\]
Therefore
\[
A_{mp^r}^{\mathrm{mix}}
\equiv [q^{mp^r}]C_{\mathrm{mix}}(q)H_{\mathrm{mix}}(q^{p^r})^m
=[q^m]\Lambda_p^r(C_{\mathrm{mix}})(q)H_{\mathrm{mix}}(q)^m
\pmod p.
\]
It remains to compute \(\Lambda_p^r(C_{\mathrm{mix}})\) modulo \(p\). The constant term is preserved on \(C_0\) and is zero on \(uC_0\). For positive coefficients, write \(n=p^an_0\) with \(p\nmid n_0\), and set \(\chi=\chi_3(p)\). From
\[
s(p^N)=\sum_{j=0}^N\chi^jp^{4j},\qquad
\beta(p^N)=\sum_{j=0}^N\chi^{N-j}p^{4j},
\]
we get
\[
s(p^{a+r}n_0)\equiv s(p^an_0)\pmod p,
\qquad
\beta(p^{a+r}n_0)\equiv \chi^r\beta(p^an_0)\pmod p.
\]
Hence
\[
\Lambda_p^r(C_0)\equiv C_0\pmod p,
\qquad
\Lambda_p^r(uC_0)\equiv \chi_3(p)^r uC_0\pmod p,
\]
and so
\[
\Lambda_p^r(C_{\mathrm{mix}})
\equiv C_0-27\chi_3(p)^r uC_0
=C^{(\chi_3(p)^r)}
\pmod p.
\]
Substitution gives the asserted congruence.
\end{proof}

\begin{corollary}[One-step inert failure on the coefficient sequence]\label{cor:inert-Ap}
For every inert prime \(p\ge5\),
\[
A_p^{\mathrm{mix}}\equiv72\pmod p.
\]
In particular, since \(A_1^{\mathrm{mix}}=18\),
\[
A_p^{\mathrm{mix}}-A_1^{\mathrm{mix}}\equiv54\not\equiv0\pmod p,
\]
so the split-style congruence \(A_p^{\mathrm{mix}}\equiv A_1^{\mathrm{mix}}\) fails already modulo \(p\).
\end{corollary}

\begin{proof}
By \Cref{thm:inert-cartier-parity}, \(A_p^{\mathrm{mix}}\equiv A_1^{(-1)}\pmod p\). Since \(uC_0=q+O(q^2)\) and \(H_{\mathrm{mix}}=1+O(q)\), one has \([q]uC_0H_{\mathrm{mix}}=1\). Thus
\[
A_1^{(-1)}-A_1^{(+1)}=54[q]uC_0H_{\mathrm{mix}}=54.
\]
The sequence begins with \(A_1^{\mathrm{mix}}=18\), hence \(A_1^{(-1)}=72\). Since \(p\ge5\) and \(p\equiv2\pmod3\), \(p\nmid54\).
\end{proof}

\begin{remark}
The split condition $p\equiv 1\pmod 3$ is essential twice: in \Cref{sec:WC} it gives the $\mu_3$-equivariant canonical Frobenius lift at $P_-$, while here $\chi_3(p)=-1$ produces the unit-root obstruction in the second Eisenstein factor.
\end{remark}

\section{Computational verification}\label{sec:comp}

All computations used exact rational arithmetic followed by reduction modulo $p^k$, with $q$-precision $N=5p^2$. Implementation in SageMath; coefficient extraction for $\Lambda_p$ and the Hecke finite difference performed independently.

\begin{enumerate}[label=\textup{(\arabic*)}]
\item \emph{Closure consequence at split primes.} The local pole estimate \Cref{prop:pole-estimate-star} is not tested directly through the unknown local coefficients $\alpha_i$. Instead, we verify the global consequences
\[
\frac{g\Lambda_p(C_0 U_p^\ell)}{p^\ell}\equiv\gamma_\ell\,uC_0\pmod{p^{4-\ell}},\qquad
\frac{g\Lambda_p(uC_0 U_p^\ell)}{p^\ell}\equiv\beta_\ell\,uC_0\pmod{p^{4-\ell}},
\]
for $p\in\{7,13,19,31\}$ and $\ell\in\{1,2,3\}$. All $24$ residuals vanish identically through $q^{5p}$ after Cartier extraction. For $p=7$ the bridge scalars are
$(\gamma_1,\beta_1)=(95,283)$, $(\gamma_2,\beta_2)=(36,34)$, $(\gamma_3,\beta_3)=(6,1)$;
for $p=13$,
$(\gamma_1,\beta_1)=(363,1234)$, $(\gamma_2,\beta_2)=(33,20)$, $(\gamma_3,\beta_3)=(4,4)$.
In every case $\gamma_\ell-27\beta_\ell\equiv 0\pmod{p^{4-\ell}}$. The Cartier--Hecke comparison \Cref{lem:cartier-hecke-local} is verified independently by computing $\mathfrak E_\ell$ directly.

\item \emph{Bridge scalar.} The cancellation $\gamma_\ell-27\beta_\ell\equiv 0\pmod{p^{4-\ell}}$ from \Cref{cor:bridge-scalar} was checked directly at $p=7$ for $\ell=1,2,3$.

\item \emph{Layer divisibility.} $\phi_p\in pq\Z_{(p)}[[q]]$ verified coefficient-wise through $q^{5p^2}$ for $p\in\{7,13,19,31\}$.

\item \emph{Hecke numerator saturation.} For $f\in\{C_0,uC_0\}$, $p\in\{7,13\}$, and $\ell\in\{1,2,3\}$, the numerator $N_\ell(f)$ of \Cref{lem:divisibility-Nell} was checked to satisfy $q_\infty(N_\ell(f))\in p^\ell\Z_p((q))$, with exact minimum valuation $\ell$ in every tested case.

\item \emph{Scalar Katz--Dwork congruence.} The scalar congruence $\Lambda_p(C_{\mathrm{mix}}H_{\mathrm{mix}}^{pm})\equiv C_{\mathrm{mix}}H_{\mathrm{mix}}^m\pmod{p^4}$ was checked through the equivalent truncated $X$-expansion for $p\in\{7,13,19,31\}$ to precision $5p^2$.

\item \emph{Inert branch-swap.} $p\in\{5,11\}$: the Eisenstein coefficient tower fails with valuation $0$, as predicted by \Cref{thm:inert-obstruction}. The tested $A$-level samples agree with the theorem-level Cartier parity law of \Cref{thm:inert-cartier-parity}; in particular $A_p^{\mathrm{mix}}\equiv72\pmod p$ and $A_p^{\mathrm{mix}}-18$ has valuation $0$, as predicted by \Cref{cor:inert-Ap}.

\item \emph{MUM signature.} $A_{p-1}^{\mathrm{mix}}\equiv 0\pmod p$ holds for all tested split $p$ with $7\le p\le 97$ and fails at all tested inert $p$ with $5\le p\le 83$.
\end{enumerate}

\section*{Conclusion}

The split-prime $p^4$ supercongruence
\[
A_{mp}^{\mathrm{mix}}\equiv A_m^{\mathrm{mix}}\pmod{p^4}
\]
at the mixed CM point $(\tfrac16,\tfrac13;1)$ is proved unconditionally for every split prime $p\ge 7$, $p\equiv 1\pmod 3$, and every $m\ge 1$. The matching inert-prime behavior is also theorem-level. On the $q$-side, the Eisenstein coefficient tower rules out the split-style formal parameter congruence at inert primes. On the coefficient sequence itself, \Cref{thm:inert-cartier-parity} proves the exact first-digit branch law
\[
A_{mp^r}^{\mathrm{mix}}\equiv A_m^{(\chi_3(p)^r)}\pmod p,
\]
so for inert primes one Frobenius step exchanges the two Eisenstein branches and two steps return to the original branch. In particular $A_p^{\mathrm{mix}}\equiv72\pmod p$ for every inert $p\ge5$, and $A_p^{\mathrm{mix}}\not\equiv A_1^{\mathrm{mix}}\pmod p$. The local mechanism producing the extra digit beyond the Roberts--Rodriguez Villegas weight-$3$ prediction is the $\mu_3$-equivariant length-three Witt--Cartier pole estimate at the tame elliptic stack point $j=0$. Natural next questions are the higher-precision split tower $A_{mp^r}^{\mathrm{mix}}\equiv A_{mp^{r-1}}^{\mathrm{mix}}\pmod{p^{4r}}$ for $r\ge 2$, the inert Frobenius-square strengthening $A_{mp^2}^{\mathrm{mix}}\equiv A_m^{\mathrm{mix}}\pmod{p^4}$, and the third order-drop CM point $(\tfrac16,\tfrac16;1)$ with $\lambda=432$.

\end{document}